\documentclass[10pt, a4paper]{article}

\usepackage{amsmath, amssymb, amsfonts, amsthm}
\usepackage{graphicx}
\usepackage{url}

\usepackage{times}
\usepackage{epsfig}
\usepackage{graphicx}
\usepackage{graphics}
\usepackage{wrapfig}
\usepackage{dsfont}
\usepackage{xspace}
\usepackage{color}
\usepackage{enumitem}
\usepackage{algorithm2e}
\usepackage{subfigure}
\usepackage{thmtools}
\usepackage{tikz}
\usepackage{lscape}
\usepackage[hypertexnames=false]{hyperref}
\usepackage{wrapfig}
\usepackage{mathtools}
\usepackage{latexsym}
\usepackage{authblk}

\newcommand{\image}{u}

\newcommand{\WEner}{{\mathcal{W}}}

\newcommand{\Econt}{{\mathcal{E}}}
\newcommand{\Ediscr}{{\mathbf{E}}}

\newcommand{\R}{{\mathbb{R}}}

\renewcommand{\d}{{\,\mathrm{d}}}

\newcommand{\tr}{{\mathrm{tr}}}

\newcommand{\compdom}{{D}}

\newcommand\restr[2]{{\left.\kern-\nulldelimiterspace #1 \vphantom{\big|}\right|_{#2}}}

\newcommand{\manifold}{{\mathcal M}}

\newcommand{\notinclude}[1]{}

\newcommand{\interpol}{{\mathcal I}}
\newcommand{\Bezier}{{\mathcal B}}
\newcommand{\Interpol}{{\mathbf I}}
\newcommand{\BEzier}{{\mathbf B}}

\def\XXint#1#2#3{{\setbox0=\hbox{$#1{#2#3}{\int}$}
\vcenter{\hbox{$#2#3$}}\kern-.5\wd0}}

\newcommand{\beq}{\begin{equation*}}
\newcommand{\eeq}{\end{equation*}}
\newcommand{\beqn}{\begin{equation}}
\newcommand{\eeqn}{\end{equation}}
\newcommand{\beqa}{\begin{eqnarray*}}
\newcommand{\eeqa}{\end{eqnarray*}}
\newcommand{\beqan}{\begin{eqnarray}}
\newcommand{\eeqan}{\end{eqnarray}}

\newtheorem{theorem}{Theorem}[section]

\newtheorem{remark}{Remark}[theorem]

\def\onedot{.}
 \def\Eg{\emph{E.g}\onedot}
\def\ie{\emph{i.e}\onedot} 
\def\cf{\emph{cf}\onedot}

\def\etal{\emph{et al}\onedot}

\makeatother

\begin{document}

\title{B\'{e}zier curves in the space of images}

\author{
 Alexander Effland, Martin Rumpf, Stefan Simon, Kirsten Stahn, Benedikt Wirth
}

\maketitle

\begin{abstract}
B\'{e}zier curves are a widespread tool for the design of curves in Euclidian space.
This paper generalizes the notion of B\'{e}zier curves to the infinite-dimensional space of images.
To this end the space of images is equipped with a Riemannian metric which measures
the cost of image transport and intensity variation in the sense of the metamorphosis model \cite{MiYo01}.
B\'ezier curves are then computed via the Riemannian version of de Casteljau's algorithm, which is based on a 
hierarchical scheme of convex combination along geodesic curves. 
Geodesics are approximated using a variational discretization of the Riemannian path energy.
This leads to a generalized de Casteljau method to compute suitable discrete B\'{e}zier curves in image space.
Selected test cases demonstrate qualitative properties of the approach. 
Furthermore, a B\'{e}zier  approach for the modulation of face interpolation and 
shape animation via image sketches is presented. 
\end{abstract}

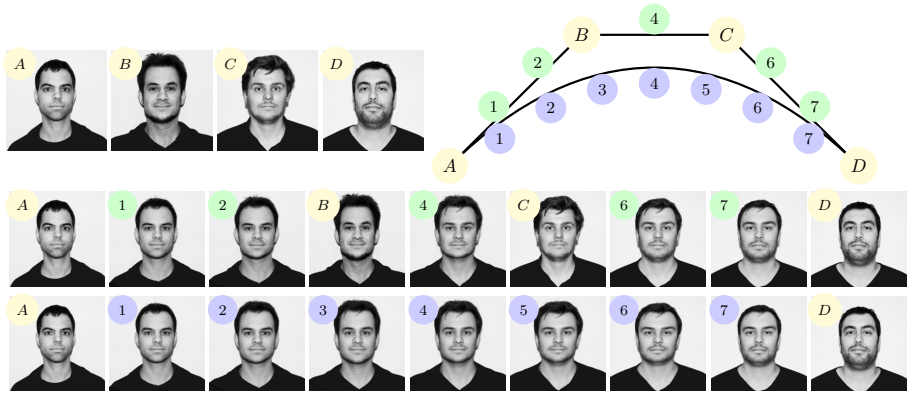
\begin{figure}[t]
\centering
\resizebox{1.0\linewidth}{!}{
\begin{tikzpicture}[x=1cm, y=1cm]
\draw[-, very thick] (6,0.5) --(8.5,3) -- (11.25,3) -- (13.75,0.5);
  
\draw (6,0.5) coordinate (A);
\draw (8.5,3) coordinate (B);
\draw (11.25,3) coordinate (C);
\draw (13.75,0.5) coordinate (D);
\draw [very thick](A) .. controls (B) and (C) .. (D);
\draw (6,0.5) node[circle, minimum size = 0.5cm, fill=yellow!20] {$ A $};
\draw (8.5,3) node[circle, minimum size = 0.5cm, fill=yellow!20] {$ B $};
\draw (11.25,3) node[circle, minimum size = 0.5cm, fill=yellow!20] {$ C $};
\draw (13.75,0.5) node[circle, minimum size = 0.5cm, fill=yellow!20] {$ D $};
 \draw (6.9482421875,1.3203125)[below] node[circle, minimum size = 0.05cm, fill=blue!20] {\small$ 1 $};
 \draw (7.9140625,1.90625)[below] node[circle, minimum size = 0.05cm, fill=blue!20] {\small$2 $};
 \draw (8.8916015625,2.2578125)[below] node[circle, minimum size = 0.05cm, fill=blue!20] {\small$3 $};
 \draw (9.875,2.375)[below] node[circle, minimum size = 0.05cm, fill=blue!20] {\small$4 $};
 \draw (10.8583984375,2.2578125)[below] node[circle, minimum size = 0.05cm, fill=blue!20] {\small$5 $};
 \draw (11.8359375,1.90625)[below] node[circle, minimum size = 0.05cm, fill=blue!20] {\small$6 $};
 \draw (12.8017578125,1.3203125)[below] node[circle, minimum size = 0.05cm, fill=blue!20] {\small$7 $};
  
 \draw (6.8333333,4/3)[above] node[circle, minimum size = 0.05cm, fill=green!20] {\small$1 $};
 \draw (23/3,2.16666666)[above] node[circle, minimum size = 0.05cm, fill=green!20] {\small$2 $};
 \draw (9.875,3)[above] node[circle, minimum size = 0.05cm, fill=green!20] {\small$4 $};
 \draw (12.08333333,2.16666666)[above] node[circle, minimum size = 0.05cm, fill=green!20] {\small$6 $};
 \draw (12.916666666,4/3)[above] node[circle, minimum size = 0.05cm, fill=green!20] {\small$7 $};
 
\draw (-1.45 , 1.75 ) node [align = center] {\includegraphics[width=1.9cm, keepaspectratio]{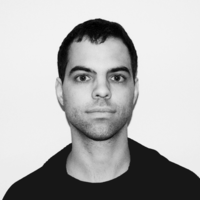}};
\draw (0.55, 1.75) node [align = center] {\includegraphics[width=1.9cm, keepaspectratio]{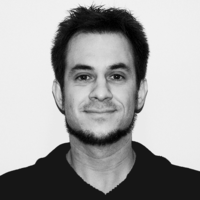}};
\draw (2.55, 1.75) node [align = center] {\includegraphics[width=1.9cm, keepaspectratio]{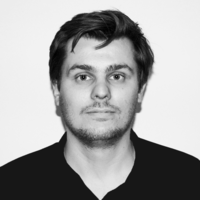}};
\draw (4.55, 1.75) node [align = center] {\includegraphics[width=1.9cm, keepaspectratio]{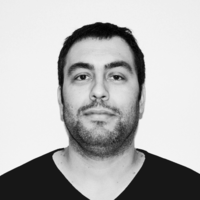}};
\draw (-1.45,-0.9) node [align = center] {\includegraphics[width=1.8cm, keepaspectratio]{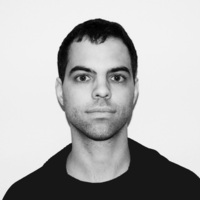}};
\draw (0.45,-0.9) node [align = center] {\includegraphics[width=1.8cm, keepaspectratio]{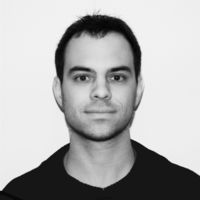}};
\draw (2.35,-0.9) node [align = center] {\includegraphics[width=1.8cm, keepaspectratio]{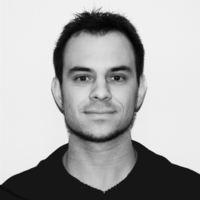}};
\draw (4.25,-0.9) node [align = center] {\includegraphics[width=1.8cm, keepaspectratio]{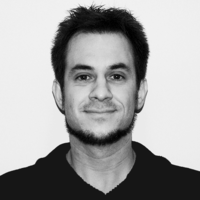}};
\draw (6.15,-0.9) node [align = center] {\includegraphics[width=1.8cm, keepaspectratio]{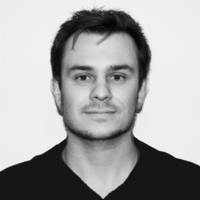}};
\draw (8.05,-0.9) node [align = center] {\includegraphics[width=1.8cm, keepaspectratio]{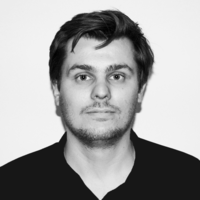}};
\draw (9.95,-0.9) node [align = center] {\includegraphics[width=1.8cm, keepaspectratio]{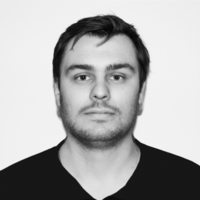}};
\draw (11.85,-0.9) node [align = center] {\includegraphics[width=1.8cm, keepaspectratio]{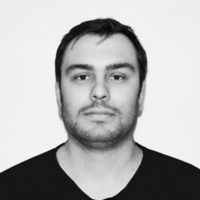}};
\draw (13.75,-0.9) node [align = center] {\includegraphics[width=1.8cm, keepaspectratio]{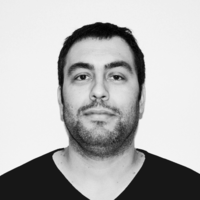}};
\draw (-1.45,-2.9) node [align = center] {\includegraphics[width=1.8cm, keepaspectratio]{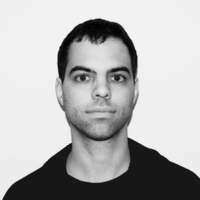}};
\draw (0.45,-2.9) node [align = center] {\includegraphics[width=1.8cm, keepaspectratio]{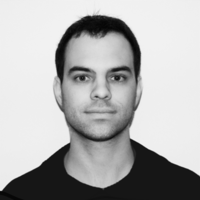}};
\draw (2.35,-2.9) node [align = center] {\includegraphics[width=1.8cm, keepaspectratio]{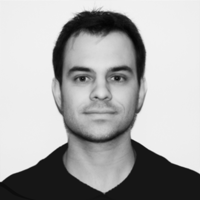}};
\draw (4.25,-2.9) node [align = center] {\includegraphics[width=1.8cm, keepaspectratio]{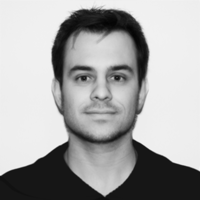}};
\draw (6.15,-2.9) node [align = center] {\includegraphics[width=1.8cm, keepaspectratio]{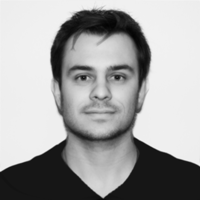}};
\draw (8.05,-2.9) node [align = center] {\includegraphics[width=1.8cm, keepaspectratio]{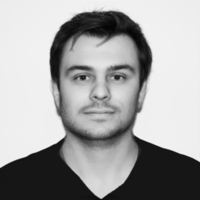}};
\draw (9.95,-2.9) node [align = center] {\includegraphics[width=1.8cm, keepaspectratio]{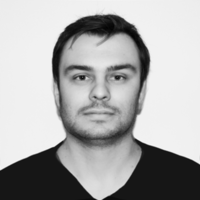}};
\draw (11.85,-2.9) node [align = center] {\includegraphics[width=1.8cm, keepaspectratio]{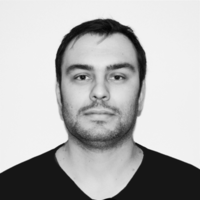}};
\draw (13.75,-2.9) node [align = center] {\includegraphics[width=1.8cm, keepaspectratio]{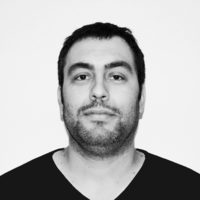}};
\draw (-2.15,2.45) node[circle, minimum size = 0.05cm, fill=yellow!20] {\footnotesize$ A $};
\draw (-0.15,2.45) node[circle, minimum size = 0.05cm, fill=yellow!20] {\footnotesize$ B $};
\draw (1.85,2.45) node[circle, minimum size = 0.05cm, fill=yellow!20] {\footnotesize$ C $};
\draw (3.85,2.45) node[circle, minimum size = 0.05cm, fill=yellow!20] {\footnotesize$ D $};
 \draw (-2.1,-0.25) node[circle, minimum size = 0.05cm, fill=yellow!20] {\footnotesize$ A $};
 \draw (-0.2,-0.25) node[circle, minimum size = 0.05cm, fill=green!20] {\footnotesize$ 1 $};
 \draw (1.7,-0.25) node[circle, minimum size = 0.05cm, fill=green!20] {\footnotesize$2 $};
 \draw (3.6,-0.25) node[circle, minimum size = 0.05cm, fill=yellow!20] {\footnotesize$B $};
 \draw (5.5,-0.25) node[circle, minimum size = 0.05cm, fill=green!20] {\footnotesize$ 4 $};
 \draw (7.4,-0.25) node[circle, minimum size = 0.05cm, fill=yellow!20] {\footnotesize$C $};
 \draw (9.3,-0.25) node[circle, minimum size = 0.05cm, fill=green!20] {\footnotesize$6 $};
 \draw (11.2,-0.25) node[circle, minimum size = 0.05cm, fill=green!20] {\footnotesize$7 $};
 \draw (13.1,-0.25) node[circle, minimum size = 0.05cm, fill=yellow!20] {\footnotesize$ D $};
 
 \draw (-2.1,-2.25) node[circle, minimum size = 0.05cm, fill=yellow!20] {\footnotesize$ A $};
 \draw (-0.2,-2.25) node[circle, minimum size = 0.05cm, fill=blue!20] {\footnotesize$1 $};
 \draw (1.7,-2.25) node[circle, minimum size = 0.05cm, fill=blue!20] {\footnotesize$2 $};
 \draw (3.6,-2.25) node[circle, minimum size = 0.05cm, fill=blue!20] {\footnotesize$3 $};
 \draw (5.5,-2.25) node[circle, minimum size = 0.05cm, fill=blue!20] {\footnotesize $4$};
 \draw (7.4,-2.25) node[circle, minimum size = 0.05cm, fill=blue!20] {\footnotesize$5 $};
 \draw (9.3,-2.25) node[circle, minimum size = 0.05cm, fill=blue!20] {\footnotesize$6 $};
 \draw (11.2,-2.25) node[circle, minimum size = 0.05cm, fill=blue!20] {\footnotesize$7 $};
 \draw (13.1,-2.25) node[circle, minimum size = 0.05cm, fill=yellow!20] {\footnotesize$ D $};
\end{tikzpicture}
}
\caption{Given 4 portrait images as control points (A--D) a discrete B\'ezier curves in the metamorphosis manifold (marked in blue)
are compared with the geodesic interpolation (marked in green).
}
\label{fig:teaser}
\end{figure}

\section{Introduction}
B\'ezier curves are a classical tool in computer aided geometric design. 
We generalize this tool to an infinite-dimensional manifold of images. 
To this end we consider images as control points and compute a curve in the space of images by a Riemannian version of de Casteljau's algorithm.
Thus, our approach relies on the definition of a Riemannian structure and the resulting notion of geodesic curves on the space of images.

The study of the space of images from a Riemannian manifold perspective allows to transfer various tools  
from classical geometry and computer-aided design to this infinite-dimensional space. \Eg, geodesics as minimizers of the path length
generalize the notion of straight lines from Euclidian space, the geodesic distance allows to measure dissimilarity of images in a rigorous way,
the geometric exponential map allows to generate natural extrapolation paths starting from an infinitesimal variation of an image.
During the past decade, this geometric approach triggered the development of new methods in 
computer vision and imaging, ranging from shape statistics \cite{FlLuPi04} to computational anatomy \cite{BeMiTrYo02}. 

Different image manifolds have been investigated. The concept of optimal transport was used to study the space of images, where image intensity functions are considered as probability measures \cite{BeBr00,ZhYaHa07}.
Following the classical paradigm by Arnold \cite{Ar66a,ArKh98}, the temporal change of images can be studied based on the flow of diffeomorphism concept, where  a family of diffeomorphisms $(\psi(t))_{t\in [0,1]}: \bar D \to \R^d$ on $\bar \compdom\subset\R^d$ describes a flow transporting image intensities along particle paths. The metamorphosis approach was first proposed by Miller and Younes \cite{MiYo01} as a generalization of the flow of diffeomorphism model and comprehensively analyzed by Trouv\'e and Younes \cite{TrYo05}. Besides pure transport it allows in addition for image intensity variations along motion paths. 

On any Riemannian image manifold geodesics are not only minimizers of the length functional but also 
of the path energy, which is the time integral of the squared path velocity. In addition, minimizers of the path energy 
have constant speed parameterization. Hence, these geodesics are not only the obvious generalization of straight lines in Euclidian space,
but also allow a simple procedure to compute convex combinations with convex coefficients $t$ and $1-t$: one simply evaluates the path in image space at time $t\ell$, where $\ell$ is the total path length.
Convex combination is the key ingredient of de Casteljau's algorithm. This is our starting point, and in 
this paper we will focus on the metamorphosis model and use the above insight to study B\'ezier curves in this space. 

Already in 1995, Park and Ravani realized how de Casteljau's classical algorithm can be applied on Riemannian manifolds
(just replacing straight lines by geodesics) to define smooth B\'ezier curves,
and they applied this concept to generate smooth trajectories for moving objects on a manifold of kinematically admissible motions \cite{PaRa95}.
Such B\'ezier curve segments on manifolds can also be patched together to yield a spline interpolation of points on the manifold.
To this end, Popiel and Noakes derived formulae for the endpoint velocities and (covariant) accelerations of B\'ezier curves in terms of the B\'ezier control points \cite{PoNo07},
which can be employed to ensure smoothness of up to two derivatives.
In \cite{MoCaVe08}, a concept of geodesics from discrete differential geometry is used to apply de Casteljau's algorithm for the purpose of curve modeling on triangulated surfaces.
$C^1$ smoothness of B\'ezier splines is here ensured by extrapolating the endpoint velocity of a B\'ezier patch
via so-called straightest geodesics to obtain control points of the next patch.
The authors also implement a classical control point refinement procedure to approximate B\'ezier curves via a subdivision scheme.
Gousenbourger et al.\ compute interpolating $C^1$ splines on the sphere, $SO(3)$, and a shape manifold of closed curves \cite{GoSaAb14},
where they generalize the idea of cubic B\'ezier splines with minimum acceleration to manifolds.
Let us finally mention that there are also alternative generalizations of polynomial curves to manifolds based on variational definitions.
In particular, generalized cubic curves can be defined as minimizers of the integrated squared second covariant derivative \cite{CaSiCr01,PoNo07a}.

In this paper we discuss B\'ezier curves in the space of images considered as a Riemannian manifold using the metamorphosis approach.
Figure \ref{fig:teaser} shows the comparison of piecewise geodesic interpolation and a discrete cubic B\'ezier curve for four different facial images as control points---from now on called control images. The B\'ezier curve is temporally smooth, comes close to the two intermediate images, and features of the control images globally pervade the curve.

\section{Time Discrete Metamorphosis}
In this section, we briefly review the metamorphosis model and introduce a suitable variational time discretization. With the notion of continuous and discrete path energy at hand we can define a continuous and discrete geodesic interpolation.

\subsection{Continuous Model}
The Riemannian metric in the metamorphosis model combines a measurement of friction 
caused by the flow of the image intensities and a measure of the variation of intensity values along motion paths. More explicitly, along a path $u: [0,1] \to L^2(D)$ in the manifold of images on a bounded Lipschitz domain $D$ we consider the Riemannian metric
\beqn
g_{u(t)}(\dot u(t),\dot u(t)) = \min_{v} \int_D L[v(t),v(t)] + \frac1\delta \left(\frac{D}{\partial t} u(t)\right)^2 \d x\,,
\eeqn
where $\dot u(t)$ represents the velocity along the path (and is just the pointwise time derivative of the image $u(t)$).
The first term describes the viscous dissipation in a multipolar fluid model (\cf\ \cite{NeSi91}).
Here, we assume that velocities field vanish on $\partial D$ and consider 
$L[v,v] :=  \tfrac{\lambda}{2} (\tr\varepsilon[v])^2+ \mu\tr(\varepsilon[v]^2) + \gamma |D^m v|^2$ 
as the classical viscous dissipation model for a Newtonian flow plus a simple multipolar dissipation model, where $\varepsilon[v]= \frac12(\nabla v +  \nabla v^T)$, $m>1+\frac{d}{2}$ (for $d$ the space dimension) and $\lambda, \, \mu,\, \gamma >0$.
The second term with weight $\frac1\delta>0$ measures the temporal variation of the intensity in terms of the material derivative 
$\frac{D}{\partial t} u = \dot u + \nabla u \cdot v$. Obviously, the same temporal change $\dot u(t)$ in the image intensity can be produced by different motion fields $v(t)$ and associated material derivatives $\frac{D}{\partial t} u$, which makes the minimization with respect to $v$ in the definition of the metric necessary.
The path energy is given by
\beqn\label{eq:DefinitionPathenergy}
\Econt[(u(t))_{t\in [0,1]}] = \int_0^1 g_{u(t)}(\dot u(t),\dot u(t)) \d t\,.
\eeqn
Dupuis \etal\ \cite{DuGrMi98} showed already for the pure transport model ($\delta=\infty$) that paths 
of finite energy are indeed one-parameter families of diffeomorphisms. 
A rigorous analytical treatment of the general model can be found in \cite{TrYo05a} including the existence of minimizing paths connecting two images in $L^2(D)$.
It relies on an appropriate notion of tangent vectors, which are equivalence classes of pairs $(v,z)$, where $v$ is a motion field with square integrable Jacobian, 
and $z$ is a weak representation of the material derivative $\frac{D}{\partial t} u$.
Let $(u(t))_{t\in [0,1]}$ denote a continuous geodesic curve in the space of images given as the minimizer of the 
path energy for given images $u(0) = u_A$ and $u(1)=u_B$,
then we can define a convex combination of the images $u_A$ and $u_B$ with weight $\lambda$
via an interpolation of the geodesic at time $t=\lambda$, \ie
\beqn
\interpol(u_A,u_B,\lambda)=u(\lambda)\,.
\eeqn
Here, we in particular make use of the fact that the ``speed'' $\sqrt{g_{u(t)}(\dot u(t),\dot u(t))}$ is constant for a minimizer of the path energy.

\subsection{Time Discrete Model}
To robustly and efficiently approximate geodesic curves in the metamorphosis model 
we define a suitable time discrete version of the continuous 
path energy \eqref{eq:DefinitionPathenergy}. To this end, we consider $(K+1)$-tuples 
$(u_0,\ldots, u_K)$ of images and define on them an energy
\beqn
\Ediscr[(u_0,\ldots, u_K)] = K \sum_{k=1}^K \WEner[u_{k-1},u_k]\,,
\eeqn
where $\WEner[u_{k-1},u_k]$ is a classical matching functional measuring the cost to match the image 
$u_{k-1}$ with the image $u_k$. In detail,
\begin{equation}
\WEner[u,\tilde u]= \min_{\phi} \int_{D} W(D \phi) + \gamma |D^m \phi|^2+
\frac{1}{\delta} |\tilde u \circ \phi - u|^2 \d x
\label{eq:WEnerDefinition}
\end{equation}
for an isotropic and rigid body motion invariant energy density $W$. A suitable choice in the case $d=2$ is given in \cite{TimeDiscrete}.
In our numerical computations we will use the simplified energy $W(D\phi)= |D \phi|^2$ and we will replace $\gamma |D^m \phi|^2$ by $\gamma |\triangle \phi|^2$ corresponding to the quadratic form $L[v(t),v(t)] = Dv:Dv + \gamma \triangle v \cdot \triangle v$.
Minimization is performed over a set of admissible deformations with $\phi(x) = x$ on $\partial D$. 
A \emph{discrete geodesic path} is defined as a minimizer of the 
discrete path energy $\Ediscr$ for $u_0 = u_A$ and $u_K = u_B$.
In this discrete path energy, two opposing effects can be observed. The last term in \eqref{eq:WEnerDefinition} penalizes intensity variations along the discrete motion path $(x, \phi_1(x), (\phi_2\circ \phi_1)(x), \ldots, (\phi_K \circ \ldots \circ \phi_1)(x))$, 
whereas the first two terms penalize deviations of the (discrete) flow along these discrete motion paths 
from rigid body motions. In fact,  $K (\image_k \circ \phi_k-\image _{k-1})$  plays the role of a time discrete material derivative along the above discrete motion path.
This ansatz ensures that the discrete energy $\Ediscr$ 
$\Gamma$--converges for $K\to\infty$ to the continuous energy $\Econt$. For the proof we refer to  \cite{TimeDiscrete}.
In particular, discrete geodesics converge to continuous geodesics for $K\to\infty$.
An introduction to variational time discretization on shape manifolds can be found in \cite{RuWi12b}. 
Obviously, the energy scales quadratically in the deformation $\phi$, 
which itself is expected to scale linearly in the time step $\tau =\frac1K$. This motivates the coefficient $K$ in front of the discrete path energy. 
Now, assuming uniqueness of discrete geodesics a discrete geodesic interpolation in analogy to the continuous interpolation is defined as
\beqn
\Interpol^K(u_A,u_B,k)=u_k
\eeqn
for the discrete geodesic path $(u_0,\ldots, u_K)$.

\section{De Casteljau's Algorithm on the Manifold of Images}
A  B\'ezier curve is a polynomial curve of degree $n>1$ in Euclidean space $\R^d$
that is defined by $n+1$ control points $x_0,\ldots,x_n\in\R^d$.
The curve emanates from the first and ends in the last control point,
but does not interpolate the intermediate control points.
However, it roughly follows those points and always remains within their convex hull.
From another viewpoint, the choice of the $n+1$ control points
is equivalent to specifying the first $\frac{n+1}2$ derivatives at the curve end points.
The classical de Casteljau algorithm for evaluating points on a B\'ezier curve
can be used to generalize those curves from $\R^n$ to Riemannian manifolds \cite{PaRa95}.

\subsection{Continuous B\'ezier curves}
De Casteljau's algorithm constructs each point along a B\'ezier curve recursively
via an iteration of weighted interpolations between point pairs,
starting from weighted interpolations between all neighboring control points.
Since weighted interpolation between points can also be performed on Riemannian manifolds,
the algorithm can directly be applied to control points on a manifold, yielding a generalisation of B\'ezier curves.
Note that on a Riemannian manifold, weighted interpolation between two points yields a point along the connecting geodesic
as opposed to a point along the connecting straight line in Euclidean space.

In our case, the control points are images $u_0^0,\ldots,u_n^0:D\to\R$,
and the corresponding B\'ezier curve\notinclude{ $\Bezier(u_0^0,\ldots,u_n^0,\cdot):[0,1]\to\manifold$,} $t\mapsto \Bezier(u_0^0,\ldots,u_n^0,t)$
at any time $t\in [0,1]$ is recursively defined via de Casteljau's algorithm as 
$$\Bezier(u_i,\ldots,u_j,t) = \interpol(\Bezier(u_i,\ldots,u_{j-1},t), \Bezier(u_{i+1},\ldots,u_{j},t), t)\,.$$ 
Geodesics are not unique in general, hence the interpolations and the resulting B\'ezier curve for a given set of input images are in general non-unique as well.
We obtain the following existence and continuity result.
\begin{theorem}[Existence and stability of B{\'e}zier curves]\label{thm1}
Let $D$ be a Lipschitz domain and $n\geq1$.
For any $n+1$ input images $u_0^0,\ldots,u_n^0 \in L^2(D)$ there exists a B{\'e}zier curve $\Bezier(u_0^0,\ldots,u_n^0,\cdot):[0,1]\to L^2(D)$.
Furthermore, if $(u_{0,k}^0,\ldots,u_{n,k}^0)_{k=1,2,\ldots}$ is a sequence of control points converging against $(u_0^0,\ldots,u_n^0)$ in $(L^2(D))^{n+1}$,
then there exist B\'ezier curves $\Bezier(u_{0,k}^0,\ldots,u_{n,k}^0,\cdot)$ of which a subsequence converges pointwise against a B\'ezier curve $\Bezier(u_0^0,\ldots,u_n^0,\cdot)$.
\end{theorem}
\begin{remark}
In the case that all involved geodesics and thus all B\'ezier curves are unique, we even have convergence of the full sequence.
\end{remark}
\begin{proof}
The result automatically follows from the recursive definition of B\'ezier curves
if we can show that for every $t\in[0,1]$ the interpolation $(u_1,u_2)\mapsto\interpol(u_1,u_2,t)$ exists and
is a well-defined, continuous map from $(L^2(D))^2$ to $L^2(D)$. To keep the exposition compact, we 
argue here formally and do not expand the arguments for the proper weak notion of the material derivative.

Trouv{\'e} and Younes \cite[Theorem 6]{TrYo05a} have shown the existence of a geodesic between any $u_1,u_2\in L^2(D)$,
so for the well-definedness of $\interpol(u_1,u_2,t)$ we only need to show the well-definedness of evaluating such a geodesic at any time $t\in[0,1]$.
From \cite{TrYo05a} we know that any geodesic $u(t)_{t\in [0,1]}$ is associated
with a family $\psi(t)_{t\in [0,1]}$ of diffeomorphisms and the generating motion field $v_\psi(t)=\dot\psi(t)\circ\psi(t)^{-1}$, and $u(t)_{t\in [0,1]}$ and $\psi(t)_{t\in [0,1]}$
minimize the energy
\begin{equation*}
\tilde\Econt[\psi(t)_{t\in [0,1]},u(t)_{t\in [0,1]}]
=\int_0^1\int_DL[v_\psi(t),v_\psi(t)]+\tfrac1\delta(\dot u(t)+\nabla u(t)\cdot v_\psi(t))^2\,\d x\,\d t\,,
\end{equation*}
\ie, $\Econt[u(t)_{t\in [0,1]}]=\min_{\psi(t)_{t\in [0,1]}}\tilde\Econt[\psi(t)_{t\in [0,1]},u(t)_{t\in [0,1]}]$.
As in \cite{TimeDiscrete}, for given $u(0)$, $u(1)$, and $\psi(t)_{t\in [0,1]}$,
from \cite[Theorem 4]{TrYo05a} and \cite[Theorem 2]{TrYo05a} one obtains an explicit representation of the minimizer $u(t) = u_{\psi,u(0),u(1)}(t)$
of $\tilde\Econt[\psi(t)_{t\in [0,1]},\cdot]$ with
\begin{equation*}
u_{\psi,u(0),u(1)}(t) := \left[u(0)  + (u(1) \circ \psi(1) - u(0))\frac{\int_0^t (\det D \psi)^{-1}(s) \d s}{\int_0^1 (\det D \psi)^{-1}(s) \d s}\right]\circ\psi(t)^{-1}\,.
\end{equation*}
Thus, taking $u(0)=u_1$, $u(1)=u_2$, and $\psi$ the diffeomorphism family associated with the geodesic,
the evaluation of the geodesic at $t$ is given by $u_{\psi,u_1,u_2}(t)\in L^2(D)$.

As for the continuity of $(u_1,u_2)\mapsto\interpol(u_1,u_2,t)$,
consider a sequence $(u_{1,k},u_{2,k})$ converging to $(u_1,u_2)$ in $(L^2(D))^2$.
Let us express the path energy in terms of the diffeomorphisms according to
\begin{align*}
\psi(t)_{t\in [0,1]}\mapsto\tilde\Econt^k[\psi(t)_{t\in [0,1]}]&:=\tilde\Econt[\psi(t)_{t\in [0,1]},u_{\psi,u_{1,k},u_{2,k}}(t)_{t\in [0,1]}]\,,\\
\psi(t)_{t\in [0,1]}\mapsto\tilde\Econt^\infty[\psi(t)_{t\in [0,1]}]&:=\tilde\Econt[\psi(t)_{t\in [0,1]},u_{\psi,u_{1},u_{2}}(t)_{t\in [0,1]}]\,.
\end{align*}
It is relatively straightforward to show that minimizers $(\psi_k(t))_{t\in [0,1]}$ of $\tilde\Econt^k$ converge against minimizers $(\psi_\infty(t))_{t\in [0,1]}$ of $\tilde\Econt^\infty$,
arguing by $\Gamma$-convergence of $\tilde\Econt^k$ to $\tilde\Econt^\infty$ and equicoerciveness of the $\tilde\Econt^k$.
Note that the paths $u_{\psi_k,u_{1,k},u_{2,k}}$ and  $u_{\psi_\infty,u_1,u_2}$ of images are geodesics.
As shown in \cite{TrYo05a}, the $\psi_k(t)_{t\in [0,1]}$ and their inverses are even uniformly bounded in $C^{0,\frac12}([0,1], C^{1,\alpha}(\bar D))$ for every $\alpha < m-1-\tfrac{d}{2}$
and thus admit strongly convergent subsequences in $C^{0,\beta}([0,1],C^{1,\alpha}(\bar D))$ for every $0<\beta<\tfrac12$ and $\alpha < m-1-\tfrac{d}{2}$ (which we again index by $k$ for simplicity).
This strong convergence of the $\psi_k$ and their inverses now implies the desired convergence $u_{\psi_k,u_{1,k},u_{2,k}}(t)\to u_{\psi_\infty,u_1,u_2}(t)=\interpol(u_1,u_2,t)$ in $L^2(D)$.
Indeed, this follows from the definition of the $u_{\psi_k,u_{1,k},u_{2,k}}$
since for any $\tilde u_k\to\tilde u$ in $L^2(D)$ and $\phi_k\to\phi$, $\phi_k^{-1}\to\phi^{-1}$ in $C^{1,\alpha}(\bar D)$ we have
$\|\tilde u_k\circ\phi_k-\tilde u\circ\phi\|_{L^2}
\leq\|\det D\phi_k^{-1}\|_{L^\infty}\|\tilde u_k-\tilde u\|_{L^2}+\|\tilde u\circ\phi_k-\tilde u\circ\phi\|_{L^2}$.
The second integral converges to $0$ because of the convergence of $\phi_k$ in $C^{0,\beta}((0,1),C^{1,\alpha}(\bar D))$.
\hfill$\Box$
\end{proof}

\subsection{Discrete B\'ezier curves}
In analogy to discrete geodesics, we define a discrete B\'ezier curve $$(\BEzier^K(u_0^0,\ldots,u_n^0,k))_{k=0,\ldots, K}$$ by replacing the continuous operations in the defining algorithm by discrete ones.
In explicit, a discrete $K$-B\'ezier curve of degree $n$ is a discrete path of images $(u_0,\ldots,u_K)$ defined by $n+1$ control points $u_0^0,\ldots,u_n^0$.
For the evaluation we replace the continuous geodesic interpolation $\interpol(u,\tilde u,t)$ by the discrete geodesic interpolation $\Interpol^K(u,\tilde u,k)$ with $t=\tfrac{k}{K}$ and obtain the recursive relation
$$\BEzier^K(u_i,\ldots,u_j,k) = \Interpol^K(\BEzier^K(u_i,\ldots,u_{j-1},k), \BEzier^K(u_{i+1},\ldots,u_{j},k), k)\,.$$  
For the actual computation we use the hierarchical algorithm (cf. Fig. \ref{alg:deCasteljau}). 

\begin{figure}
\begin{minipage}{0.48\linewidth}
\hspace*{5ex} \textbf{for} $j=1$ to $n$\\
\hspace*{7ex} \textbf{for} $i=j$ to $n$\\
\hspace*{9ex} $u^j_i=\Interpol^K(u^{j-1}_{i-1},u^{j-1}_i,k)$\\
\hspace*{7ex} \textbf{end}\\
\hspace*{5ex} \textbf{end}\\
\hspace*{5ex} $\BEzier^K(u_0^0,\ldots,u_n^0,k)=u_n^n$
\end{minipage}
\begin{minipage}{0.48\linewidth}
\resizebox{1.0\linewidth}{!}{
\begin{tikzpicture}[x=1cm, y=1cm]
\draw[-, very thick, black!50!green] (6,0.5) --(8.5,3) -- (11.25,3) -- (13.75,0.5);
  
\draw (6,0.5) coordinate (A);
\draw (8.5,3) coordinate (B);
\draw (11.25,3) coordinate (C);
\draw (13.75,0.5) coordinate (D);
\draw (6,0.5) [fill=black!50!green, black!50!green] circle [radius  = 0.1cm];;
\draw (8.5,3) [fill=black!50!green, black!50!green] circle [radius  = 0.1cm];;
\draw (11.25,3) [fill=black!50!green, black!50!green] circle [radius  = 0.1cm];;
\draw (13.75,0.5) [fill=black!50!green, black!50!green] circle [radius  = 0.1cm];;
\draw (89/14,6/7) [fill=black!50!green, black!50!green] circle [radius  = 0.1cm];;
\draw (94/14,17/14) [fill=black!50!green, black!50!green] circle [radius  = 0.1cm];;
\draw (99/14,11/7) [fill=black!50!green, black!50!green] circle [radius  = 0.1cm];;
\draw (104/14,27/14) [fill=black!50!green, black!50!green] circle [radius  = 0.1cm];;
\draw (109/14,16/7) [fill=black!50!green, black!50!green] circle [radius  = 0.1cm];;
\draw (114/14,37/14) [fill=black!50!green, black!50!green] circle [radius  = 0.1cm];;
\draw (249/28,3) [fill=black!50!green, black!50!green] circle [radius  = 0.1cm];;
\draw (260/28,3) [fill=black!50!green, black!50!green] circle [radius  = 0.1cm];;
\draw (271/28,3) [fill=black!50!green, black!50!green] circle [radius  = 0.1cm];;
\draw (282/28,3) [fill=black!50!green, black!50!green] circle [radius  = 0.1cm];;
\draw (293/28,3) [fill=black!50!green, black!50!green] circle [radius  = 0.1cm];;
\draw (304/28,3) [fill=black!50!green, black!50!green] circle [radius  = 0.1cm];;
\draw (325/28,37/14) [fill=black!50!green, black!50!green] circle [radius  = 0.1cm];;
\draw (335/28,16/7) [fill=black!50!green, black!50!green] circle [radius  = 0.1cm];;
\draw (345/28,27/14) [fill=black!50!green, black!50!green] circle [radius  = 0.1cm];;
\draw (355/28,11/7) [fill=black!50!green, black!50!green] circle [radius  = 0.1cm];;
\draw (365/28,17/14) [fill=black!50!green, black!50!green] circle [radius  = 0.1cm];;
\draw (375/28,6/7) [fill=black!50!green, black!50!green] circle [radius  = 0.1cm];;
\draw[-, very thick, orange] (104/14,27/14) --(282/28,3) -- (355/28,11/7);
\draw (104/14,27/14) [fill=orange!20, orange] circle [radius  = 0.1cm];;
\draw (282/28,3) [fill=orange!20, orange] circle [radius  = 0.1cm];;
\draw (355/28,11/7) [fill=orange!20, orange] circle [radius  = 0.1cm];;
\draw (765/98,102/49) [fill=orange!20, orange] circle [radius  = 0.1cm];;
\draw (401/49,219/98) [fill=orange!20, orange] circle [radius  = 0.1cm];;
\draw (839/98,117/49) [fill=orange!20, orange] circle [radius  = 0.1cm];;
\draw (876/98,249/98) [fill=orange!20, orange] circle [radius  = 0.1cm];;
\draw (913/98,132/49) [fill=orange!20, orange] circle [radius  = 0.1cm];;
\draw (950/98,279/98) [fill=orange!20, orange] circle [radius  = 0.1cm];;
\draw (2047/196,137/49) [fill=orange!20, orange] circle [radius  = 0.1cm];;
\draw (2120/196,127/49) [fill=orange!20, orange] circle [radius  = 0.1cm];;
\draw (2193/196,117/49) [fill=orange!20, orange] circle [radius  = 0.1cm];;
\draw (2266/196,107/49) [fill=orange!20, orange] circle [radius  = 0.1cm];;
\draw (2339/196,97/49) [fill=orange!20, orange] circle [radius  = 0.1cm];;
\draw (2412/196,87/49) [fill=orange!20, orange] circle [radius  = 0.1cm];;
\draw[-, very thick, blue] (876/98,249/98) --(2266/196,107/49);
\draw (876/98,249/98) [fill=blue!20, blue] circle [radius  = 0.1cm];;
\draw (2266/196,107/49) [fill=blue!20, blue] circle [radius  = 0.1cm];;
\draw (6389/686,122/49) [fill=blue!20, blue] circle [radius  = 0.1cm];;
\draw (6646/686,239/98) [fill=blue!20, blue] circle [radius  = 0.1cm];;
\draw (6903/686,117/49) [fill=blue!20, blue] circle [radius  = 0.1cm];;
\draw (7160/686,229/98) [fill=blue!20, blue] circle [radius  = 0.1cm];;
\draw (7417/686,16/7) [fill=blue!20, blue] circle [radius  = 0.1cm];;
\draw (7674/686,219/98) [fill=blue!20, blue] circle [radius  = 0.1cm];;
\end{tikzpicture}
}
\end{minipage}
\caption{The Discrete de Casteljau algorithm with a schematic sketch for $n=3$, $K=7$, and $k=4$ ($j=1$ green, $j=2$ orange, $j=3$ blue).}
\label{alg:deCasteljau}
\end{figure}
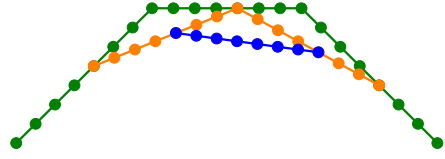
If we restrict ourselves to the space of images in $L^2(D)$ which are continuous up to a null set, we also obtain in the discrete context 
an existence and continuity result. 
\begin{theorem}[Existence and stability of discrete B{\'e}zier curves]\label{thm2}
Let $D$ be a Lipschitz domain and $n,K\geq1$.
For any $n+1$ input images $u_0^0,\ldots,u_n^0 \in L^2(D)$ there exists a discrete B{\'e}zier curve $\BEzier^K(u_0^0,\ldots,u_n^0,k)\in L^2(D)$, $k=0,\ldots,K$.
Furthermore, if $(u_{0,j}^0,\ldots,u_{n,j}^0)_{j=1,2,\ldots}$ is a sequence of control points converging against $(u_0^0,\ldots,u_n^0)$ in $(L^2(D))^{n+1}$,
then there exist B\'ezier curves $\BEzier^K(u_{0,j}^0,\ldots,u_{n,j}^0,\cdot)$ of which a subsequence converges pointwise against a B\'ezier curve $\BEzier^K(u_0^0,\ldots,u_n^0,\cdot)$.
\end{theorem}
\begin{proof}
The existence follows directly from the existence result \cite[Theorem 3.4]{TimeDiscrete} for discrete geodesics.
The proof of continuity follows the line of argumentation in the proof of Theorem\,\ref{thm1}, using compactness for the underlying family of deformations
and the fact that for two images $u_0,u_K$ and a discrete connecting geodesic with associated deformations $\phi_1,\ldots,\phi_K$,
the images $u_k=\Interpol^K(u_0,u_K,k)$ along the geodesic can be expressed as the unique solution of a block tridiagonal system of operator equations \cite[Proposition 3.2]{TimeDiscrete}
whose entries continuously depend on $\phi_1,\ldots,\phi_K$.
\hfill $\Box$
\end{proof}
Let us remark that the convergence of discrete B{\'e}zier curves for $K\to \infty$  to a continuous B{\'e}zier curve for the same set of input shapes is still open.
In fact, in \cite{TimeDiscrete} $\Gamma$--convergence of discrete geodesics is proven in the $L^2((0,1)\times D)$ topology, which is too weak 
to control the convergence of the interpolated images in the de Casteljau algorithm pointwise with respect to time.

To compute discrete geodesics numerically we have to introduce a suitable space discretization. To this end, we use a finite element discretization both for the family of images and for the family of deformations. For details we refer to \cite{TimeDiscrete}.

\begin{figure}[h]
\centering
\resizebox{1.0\linewidth}{!}{
\begin{tikzpicture}[x=1cm, y=1cm]
\draw[-, very thick] (6,0.5) --(8.5,3) -- (11.25,3) -- (13.75,0.5);
  
\draw (6,0.5) coordinate (A);
\draw (8.5,3) coordinate (B);
\draw (11.25,3) coordinate (C);
\draw (13.75,0.5) coordinate (D);
\draw [very thick](A) .. controls (B) and (C) .. (D);
\draw (6,0.5) node[circle, minimum size = 0.5cm, fill=yellow!20] {$ A $};
\draw (8.5,3) node[circle, minimum size = 0.5cm, fill=yellow!20] {$ B $};
\draw (11.25,3) node[circle, minimum size = 0.5cm, fill=yellow!20] {$ C $};
\draw (13.75,0.5) node[circle, minimum size = 0.5cm, fill=yellow!20] {$ D $};
 \draw (6.9482421875,1.3203125)[below] node[circle, minimum size = 0.05cm, fill=blue!20] {\small$ 1 $};
 \draw (7.9140625,1.90625)[below] node[circle, minimum size = 0.05cm, fill=blue!20] {\small$2 $};
 \draw (8.8916015625,2.2578125)[below] node[circle, minimum size = 0.05cm, fill=blue!20] {\small$3 $};
 \draw (9.875,2.375)[below] node[circle, minimum size = 0.05cm, fill=blue!20] {\small$4 $};
 \draw (10.8583984375,2.2578125)[below] node[circle, minimum size = 0.05cm, fill=blue!20] {\small$5 $};
 \draw (11.8359375,1.90625)[below] node[circle, minimum size = 0.05cm, fill=blue!20] {\small$6 $};
 \draw (12.8017578125,1.3203125)[below] node[circle, minimum size = 0.05cm, fill=blue!20] {\small$7 $};
  
 \draw (6.8333333,4/3)[above] node[circle, minimum size = 0.05cm, fill=green!20] {\small$1 $};
 \draw (23/3,2.16666666)[above] node[circle, minimum size = 0.05cm, fill=green!20] {\small$2 $};
 \draw (9.875,3)[above] node[circle, minimum size = 0.05cm, fill=green!20] {\small$4 $};
 \draw (12.08333333,2.16666666)[above] node[circle, minimum size = 0.05cm, fill=green!20] {\small$6 $};
 \draw (12.916666666,4/3)[above] node[circle, minimum size = 0.05cm, fill=green!20] {\small$7 $};
 
\draw (-1.45 , 1.75 ) node [align = center] {\includegraphics[width=1.9cm, keepaspectratio]{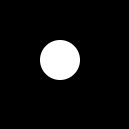}};
\draw (0.55, 1.75) node [align = center] {\includegraphics[width=1.9cm, keepaspectratio]{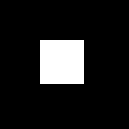}};
\draw (2.55, 1.75) node [align = center] {\includegraphics[width=1.9cm, keepaspectratio]{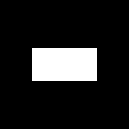}};
\draw (4.55, 1.75) node [align = center] {\includegraphics[width=1.9cm, keepaspectratio]{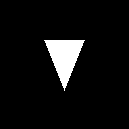}};
\draw (-1.45,-0.9) node [align = center] {\includegraphics[width=1.8cm, keepaspectratio]{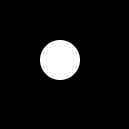}};
\draw (0.45,-0.9) node [align = center] {\includegraphics[width=1.8cm, keepaspectratio]{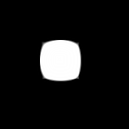}};
\draw (2.35,-0.9) node [align = center] {\includegraphics[width=1.8cm, keepaspectratio]{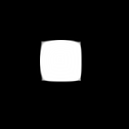}};
\draw (4.25,-0.9) node [align = center] {\includegraphics[width=1.8cm, keepaspectratio]{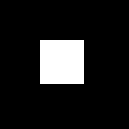}};
\draw (6.15,-0.9) node [align = center] {\includegraphics[width=1.8cm, keepaspectratio]{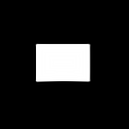}};
\draw (8.05,-0.9) node [align = center] {\includegraphics[width=1.8cm, keepaspectratio]{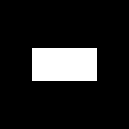}};
\draw (9.95,-0.9) node [align = center] {\includegraphics[width=1.8cm, keepaspectratio]{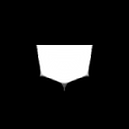}};
\draw (11.85,-0.9) node [align = center] {\includegraphics[width=1.8cm, keepaspectratio]{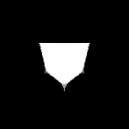}};
\draw (13.75,-0.9) node [align = center] {\includegraphics[width=1.8cm, keepaspectratio]{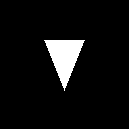}};
\draw (-1.45,-2.9) node [align = center] {\includegraphics[width=1.8cm, keepaspectratio]{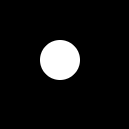}};
\draw (0.45,-2.9) node [align = center] {\includegraphics[width=1.8cm, keepaspectratio]{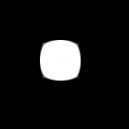}};
\draw (2.35,-2.9) node [align = center] {\includegraphics[width=1.8cm, keepaspectratio]{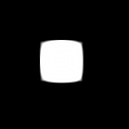}};
\draw (4.25,-2.9) node [align = center] {\includegraphics[width=1.8cm, keepaspectratio]{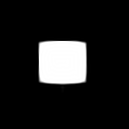}};
\draw (6.15,-2.9) node [align = center] {\includegraphics[width=1.8cm, keepaspectratio]{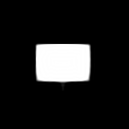}};
\draw (8.05,-2.9) node [align = center] {\includegraphics[width=1.8cm, keepaspectratio]{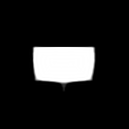}};
\draw (9.95,-2.9) node [align = center] {\includegraphics[width=1.8cm, keepaspectratio]{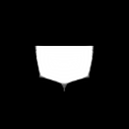}};
\draw (11.85,-2.9) node [align = center] {\includegraphics[width=1.8cm, keepaspectratio]{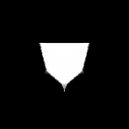}};
\draw (13.75,-2.9) node [align = center] {\includegraphics[width=1.8cm, keepaspectratio]{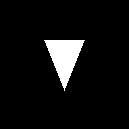}};
\draw (-2.15,2.45) node[circle, minimum size = 0.05cm, fill=yellow!20] {\footnotesize$ A $};
\draw (-0.15,2.45) node[circle, minimum size = 0.05cm, fill=yellow!20] {\footnotesize$ B $};
\draw (1.85,2.45) node[circle, minimum size = 0.05cm, fill=yellow!20] {\footnotesize$ C $};
\draw (3.85,2.45) node[circle, minimum size = 0.05cm, fill=yellow!20] {\footnotesize$ D $};
 \draw (-2.1,-0.25) node[circle, minimum size = 0.05cm, fill=yellow!20] {\footnotesize$ A $};
 \draw (-0.2,-0.25) node[circle, minimum size = 0.05cm, fill=green!20] {\footnotesize$ 1 $};
 \draw (1.7,-0.25) node[circle, minimum size = 0.05cm, fill=green!20] {\footnotesize$2 $};
 \draw (3.6,-0.25) node[circle, minimum size = 0.05cm, fill=yellow!20] {\footnotesize$B $};
 \draw (5.5,-0.25) node[circle, minimum size = 0.05cm, fill=green!20] {\footnotesize$ 4 $};
 \draw (7.4,-0.25) node[circle, minimum size = 0.05cm, fill=yellow!20] {\footnotesize$C $};
 \draw (9.3,-0.25) node[circle, minimum size = 0.05cm, fill=green!20] {\footnotesize$6 $};
 \draw (11.2,-0.25) node[circle, minimum size = 0.05cm, fill=green!20] {\footnotesize$7 $};
 \draw (13.1,-0.25) node[circle, minimum size = 0.05cm, fill=yellow!20] {\footnotesize$ D $};
 
 \draw (-2.1,-2.25) node[circle, minimum size = 0.05cm, fill=yellow!20] {\footnotesize$ A $};
 \draw (-0.2,-2.25) node[circle, minimum size = 0.05cm, fill=blue!20] {\footnotesize$1 $};
 \draw (1.7,-2.25) node[circle, minimum size = 0.05cm, fill=blue!20] {\footnotesize$2 $};
 \draw (3.6,-2.25) node[circle, minimum size = 0.05cm, fill=blue!20] {\footnotesize$3 $};
 \draw (5.5,-2.25) node[circle, minimum size = 0.05cm, fill=blue!20] {\footnotesize $4$};
 \draw (7.4,-2.25) node[circle, minimum size = 0.05cm, fill=blue!20] {\footnotesize$5 $};
 \draw (9.3,-2.25) node[circle, minimum size = 0.05cm, fill=blue!20] {\footnotesize$6 $};
 \draw (11.2,-2.25) node[circle, minimum size = 0.05cm, fill=blue!20] {\footnotesize$7 $};
 \draw (13.1,-2.25) node[circle, minimum size = 0.05cm, fill=yellow!20] {\footnotesize$ D $};
 
 \draw [red, thick] (6.15,-3.22) circle [radius=0.25cm];
 \draw [red, thick] (8.05,-3.22) circle [radius=0.25cm];
 \draw [red, thick] (9.95,-3.22) circle [radius=0.25cm];
 \draw [red, thick] (11.85,-3.22) circle [radius=0.25cm];
\end{tikzpicture}
}
\caption{Piecewise discrete geodesic (middle) and cubic B\'{e}zier curve (bottom). $K=8$, $\delta=5\cdot10^{-2}$, $\gamma=10^{-3}$.}
\label{fig:circleTriangle}
\end{figure}
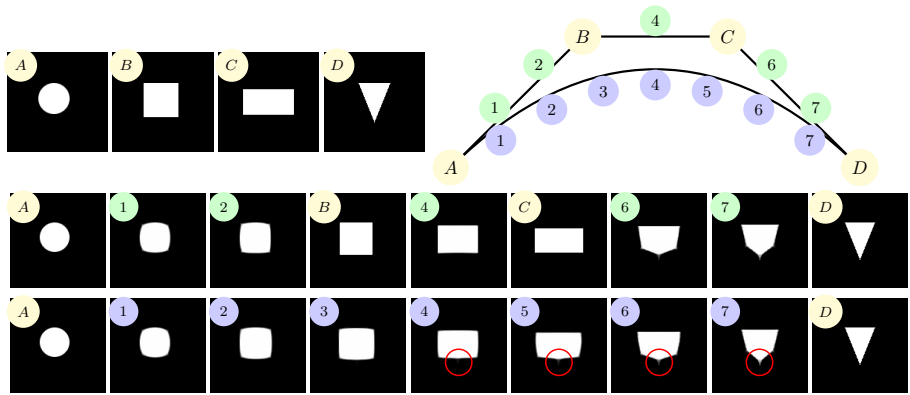

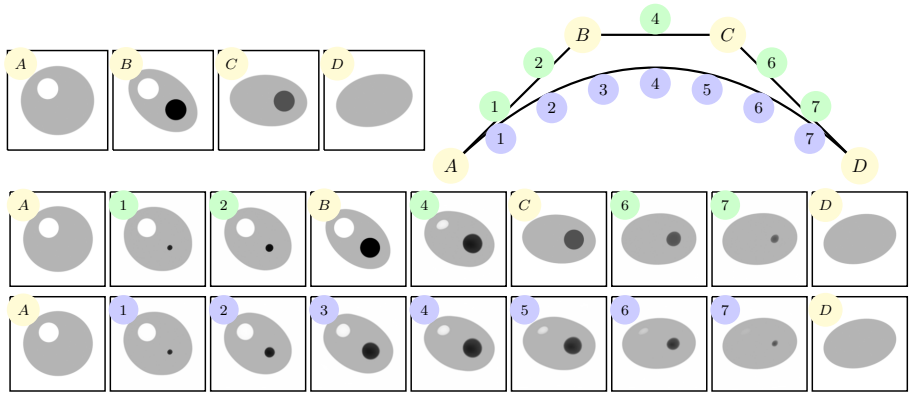
\begin{figure}[h]
\centering
\resizebox{1.0\linewidth}{!}{
\begin{tikzpicture}[x=1cm, y=1cm]
\draw[-, very thick] (6,0.5) --(8.5,3) -- (11.25,3) -- (13.75,0.5);
  
\draw (6,0.5) coordinate (A);
\draw (8.5,3) coordinate (B);
\draw (11.25,3) coordinate (C);
\draw (13.75,0.5) coordinate (D);
\draw [very thick](A) .. controls (B) and (C) .. (D);
\draw (6,0.5) node[circle, minimum size = 0.5cm, fill=yellow!20] {$ A $};
\draw (8.5,3) node[circle, minimum size = 0.5cm, fill=yellow!20] {$ B $};
\draw (11.25,3) node[circle, minimum size = 0.5cm, fill=yellow!20] {$ C $};
\draw (13.75,0.5) node[circle, minimum size = 0.5cm, fill=yellow!20] {$ D $};
 \draw (6.9482421875,1.3203125)[below] node[circle, minimum size = 0.05cm, fill=blue!20] {\small$ 1 $};
 \draw (7.9140625,1.90625)[below] node[circle, minimum size = 0.05cm, fill=blue!20] {\small$2 $};
 \draw (8.8916015625,2.2578125)[below] node[circle, minimum size = 0.05cm, fill=blue!20] {\small$3 $};
 \draw (9.875,2.375)[below] node[circle, minimum size = 0.05cm, fill=blue!20] {\small$4 $};
 \draw (10.8583984375,2.2578125)[below] node[circle, minimum size = 0.05cm, fill=blue!20] {\small$5 $};
 \draw (11.8359375,1.90625)[below] node[circle, minimum size = 0.05cm, fill=blue!20] {\small$6 $};
 \draw (12.8017578125,1.3203125)[below] node[circle, minimum size = 0.05cm, fill=blue!20] {\small$7 $};
  
 \draw (6.8333333,4/3)[above] node[circle, minimum size = 0.05cm, fill=green!20] {\small$1 $};
 \draw (23/3,2.16666666)[above] node[circle, minimum size = 0.05cm, fill=green!20] {\small$2 $};
 \draw (9.875,3)[above] node[circle, minimum size = 0.05cm, fill=green!20] {\small$4 $};
 \draw (12.08333333,2.16666666)[above] node[circle, minimum size = 0.05cm, fill=green!20] {\small$6 $};
 \draw (12.916666666,4/3)[above] node[circle, minimum size = 0.05cm, fill=green!20] {\small$7 $};
 
\draw (-1.45 , 1.75 ) node [align = center] {\includegraphics[width=1.9cm, keepaspectratio]{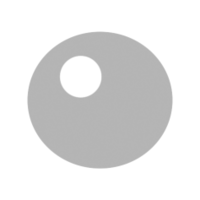}};
\draw (0.55, 1.75) node [align = center] {\includegraphics[width=1.9cm, keepaspectratio]{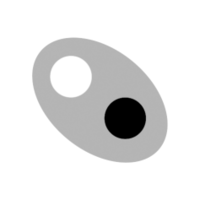}};
\draw (2.55, 1.75) node [align = center] {\includegraphics[width=1.9cm, keepaspectratio]{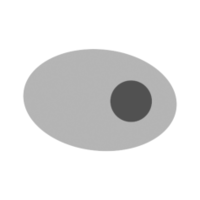}};
\draw (4.55, 1.75) node [align = center] {\includegraphics[width=1.9cm, keepaspectratio]{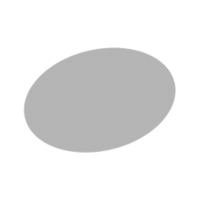}};
\draw [-, thick](-2.4,0.8) -- (-2.4,2.7) -- (-0.5, 2.7) -- (-0.5,0.8) -- (-2.4, 0.8);
\draw [-, thick](-0.4,0.8) -- (-0.4,2.7) --  (1.5, 2.7) --  (1.5,0.8) -- (-0.4, 0.8);
\draw [-, thick] (1.6,0.8) --  (1.6,2.7) --  (3.5, 2.7) --  (3.5,0.8) --  (1.6, 0.8);
\draw [-, thick] (3.6,0.8) --  (3.6,2.7) --  (5.5, 2.7) --  (5.5,0.8) --  (3.6, 0.8);
\draw (-1.45,-0.9) node [align = center] {\includegraphics[width=1.8cm, keepaspectratio]{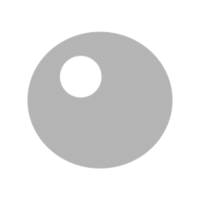}};
\draw (0.45,-0.9) node [align = center] {\includegraphics[width=1.8cm, keepaspectratio]{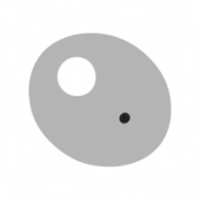}};
\draw (2.35,-0.9) node [align = center] {\includegraphics[width=1.8cm, keepaspectratio]{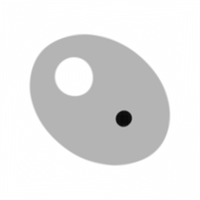}};
\draw (4.25,-0.9) node [align = center] {\includegraphics[width=1.8cm, keepaspectratio]{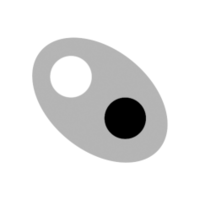}};
\draw (6.15,-0.9) node [align = center] {\includegraphics[width=1.8cm, keepaspectratio]{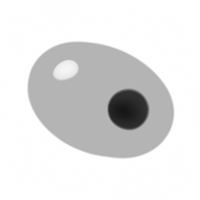}};
\draw (8.05,-0.9) node [align = center] {\includegraphics[width=1.8cm, keepaspectratio]{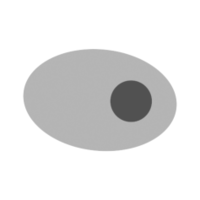}};
\draw (9.95,-0.9) node [align = center] {\includegraphics[width=1.8cm, keepaspectratio]{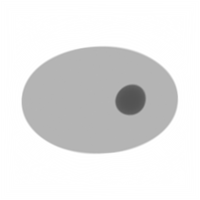}};
\draw (11.85,-0.9) node [align = center] {\includegraphics[width=1.8cm, keepaspectratio]{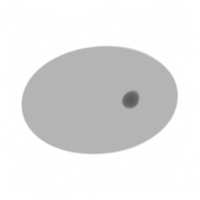}};
\draw (13.75,-0.9) node [align = center] {\includegraphics[width=1.8cm, keepaspectratio]{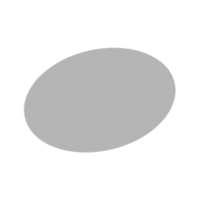}};
\draw [-, thick](-2.35,0) -- (-2.35,-1.8) -- (-0.55, -1.8) -- (-0.55,0) -- (-2.35, 0);
\draw [-, thick](-0.45,0) -- (-0.45,-1.8) --  (1.35, -1.8) --  (1.35,0) -- (-0.45, 0);
\draw [-, thick] (1.45,0) --  (1.45,-1.8) --  (3.25, -1.8) --  (3.25,0) --  (1.45, 0);
\draw [-, thick] (3.35,0) --  (3.35,-1.8) --  (5.15, -1.8) --  (5.15,0) --  (3.35, 0);
\draw [-, thick] (5.25,0) --  (5.25,-1.8) --  (7.05, -1.8) --  (7.05,0) --  (5.25, 0);
\draw [-, thick] (7.15,0) --  (7.15,-1.8) --  (8.95, -1.8) --  (8.95,0) --  (7.15, 0);
\draw [-, thick] (9.05,0) --  (9.05,-1.8) -- (10.85, -1.8) -- (10.85,0) --  (9.05, 0);
\draw [-, thick](10.95,0) -- (10.95,-1.8) -- (12.75, -1.8) -- (12.75,0) -- (10.95, 0);
\draw [-, thick](12.85,0) -- (12.85,-1.8) -- (14.65, -1.8) -- (14.65,0) -- (12.85, 0);
\draw (-1.45,-2.9) node [align = center] {\includegraphics[width=1.8cm, keepaspectratio]{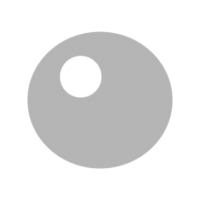}};
\draw (0.45,-2.9) node [align = center] {\includegraphics[width=1.8cm, keepaspectratio]{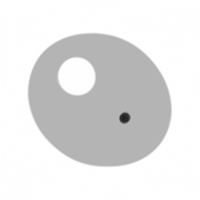}};
\draw (2.35,-2.9) node [align = center] {\includegraphics[width=1.8cm, keepaspectratio]{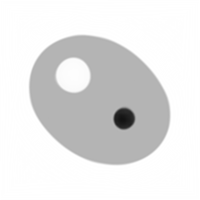}};
\draw (4.25,-2.9) node [align = center] {\includegraphics[width=1.8cm, keepaspectratio]{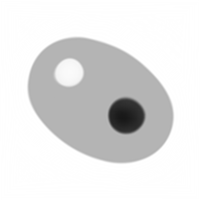}};
\draw (6.15,-2.9) node [align = center] {\includegraphics[width=1.8cm, keepaspectratio]{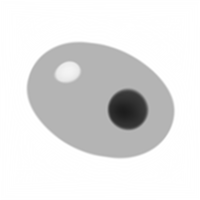}};
\draw (8.05,-2.9) node [align = center] {\includegraphics[width=1.8cm, keepaspectratio]{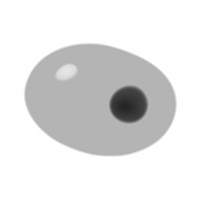}};
\draw (9.95,-2.9) node [align = center] {\includegraphics[width=1.8cm, keepaspectratio]{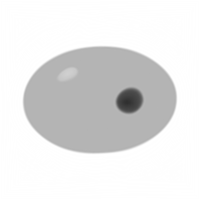}};
\draw (11.85,-2.9) node [align = center] {\includegraphics[width=1.8cm, keepaspectratio]{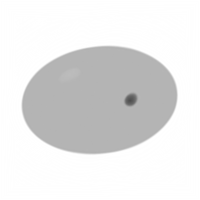}};
\draw (13.75,-2.9) node [align = center] {\includegraphics[width=1.8cm, keepaspectratio]{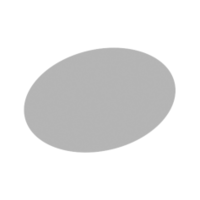}};
\draw [-, thick](-2.35,-2) -- (-2.35,-3.8) -- (-0.55, -3.8) -- (-0.55,-2) -- (-2.35, -2);
\draw [-, thick](-0.45,-2) -- (-0.45,-3.8) --  (1.35, -3.8) --  (1.35,-2) -- (-0.45, -2);
\draw [-, thick] (1.45,-2) --  (1.45,-3.8) --  (3.25, -3.8) --  (3.25,-2) --  (1.45, -2);
\draw [-, thick] (3.35,-2) --  (3.35,-3.8) --  (5.15, -3.8) --  (5.15,-2) --  (3.35, -2);
\draw [-, thick] (5.25,-2) --  (5.25,-3.8) --  (7.05, -3.8) --  (7.05,-2) --  (5.25, -2);
\draw [-, thick] (7.15,-2) --  (7.15,-3.8) --  (8.95, -3.8) --  (8.95,-2) --  (7.15, -2);
\draw [-, thick] (9.05,-2) --  (9.05,-3.8) -- (10.85, -3.8) -- (10.85,-2) --  (9.05, -2);
\draw [-, thick](10.95,-2) -- (10.95,-3.8) -- (12.75, -3.8) -- (12.75,-2) -- (10.95, -2);
\draw [-, thick](12.85,-2) -- (12.85,-3.8) -- (14.65, -3.8) -- (14.65,-2) -- (12.85, -2);
\draw (-2.15,2.45) node[circle, minimum size = 0.05cm, fill=yellow!20] {\footnotesize$ A $};
\draw (-0.15,2.45) node[circle, minimum size = 0.05cm, fill=yellow!20] {\footnotesize$ B $};
\draw (1.85,2.45) node[circle, minimum size = 0.05cm, fill=yellow!20] {\footnotesize$ C $};
\draw (3.85,2.45) node[circle, minimum size = 0.05cm, fill=yellow!20] {\footnotesize$ D $};
 \draw (-2.1,-0.25) node[circle, minimum size = 0.05cm, fill=yellow!20] {\footnotesize$ A $};
 \draw (-0.2,-0.25) node[circle, minimum size = 0.05cm, fill=green!20] {\footnotesize$ 1 $};
 \draw (1.7,-0.25) node[circle, minimum size = 0.05cm, fill=green!20] {\footnotesize$2 $};
 \draw (3.6,-0.25) node[circle, minimum size = 0.05cm, fill=yellow!20] {\footnotesize$B $};
 \draw (5.5,-0.25) node[circle, minimum size = 0.05cm, fill=green!20] {\footnotesize$ 4 $};
 \draw (7.4,-0.25) node[circle, minimum size = 0.05cm, fill=yellow!20] {\footnotesize$C $};
 \draw (9.3,-0.25) node[circle, minimum size = 0.05cm, fill=green!20] {\footnotesize$6 $};
 \draw (11.2,-0.25) node[circle, minimum size = 0.05cm, fill=green!20] {\footnotesize$7 $};
 \draw (13.1,-0.25) node[circle, minimum size = 0.05cm, fill=yellow!20] {\footnotesize$ D $};
 
 \draw (-2.1,-2.25) node[circle, minimum size = 0.05cm, fill=yellow!20] {\footnotesize$ A $};
 \draw (-0.2,-2.25) node[circle, minimum size = 0.05cm, fill=blue!20] {\footnotesize$1 $};
 \draw (1.7,-2.25) node[circle, minimum size = 0.05cm, fill=blue!20] {\footnotesize$2 $};
 \draw (3.6,-2.25) node[circle, minimum size = 0.05cm, fill=blue!20] {\footnotesize$3 $};
 \draw (5.5,-2.25) node[circle, minimum size = 0.05cm, fill=blue!20] {\footnotesize $4$};
 \draw (7.4,-2.25) node[circle, minimum size = 0.05cm, fill=blue!20] {\footnotesize$5 $};
 \draw (9.3,-2.25) node[circle, minimum size = 0.05cm, fill=blue!20] {\footnotesize$6 $};
 \draw (11.2,-2.25) node[circle, minimum size = 0.05cm, fill=blue!20] {\footnotesize$7 $};
 \draw (13.1,-2.25) node[circle, minimum size = 0.05cm, fill=yellow!20] {\footnotesize$ D $};
\end{tikzpicture}
}
\caption{Piecewise discrete geodesic (middle) and cubic B\'{e}zier curve (bottom). $K=8$, $\delta=5\cdot10^{-3}$, $\gamma=10^{-3}$.}
\label{fig:holes}
\end{figure}
Fig. \ref{fig:circleTriangle}  and Fig. \ref{fig:holes}  show discrete B\'ezier curves for two different test cases to highlight some of the general characteristics of 
B\'ezier curves in the space of images equipped with the Riemannian structure of the metamorphosis approach.
In both cases the initial image at $t=0$ and the final image at $t=1$ coincide with the corresponding control images, where as for the other control images 
the cubic B\'ezier curves approximately recovers them at times $t=\tfrac13$ and $t=\tfrac23$. The B\'ezier curve is smooth in time and structures in the images are smoothly blended and transported (cf. also the accompanying video sequence). The impact of the control images is global in time. 
Thus particular features of a single control image persist over the whole time interval $(0,1)$. 
In Fig. \ref{fig:circleTriangle} this can be best observed for the rectangular outline of the second and third control images and for the lower tip of the triangle.
In Fig. \ref{fig:holes} one obtains a simultaneous rotation of the emerging grey ellipsoid and a continuous fading in and out of the white and black circle, respectively.
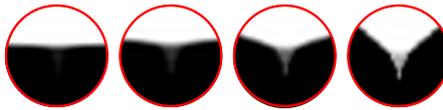
\begin{wrapfigure}{r}{.5\linewidth}
\vspace*{-3.2ex}
\centering
\resizebox{1.0\linewidth}{!}{
\begin{tikzpicture}[x=1cm, y=1cm]
\draw (0 ,0) node [align = center] {\includegraphics[width=2.4cm, keepaspectratio]{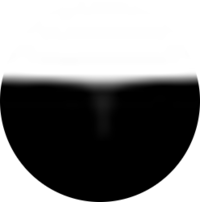}};
\draw (2.7, 0) node [align = center] {\includegraphics[width=2.4cm, keepaspectratio]{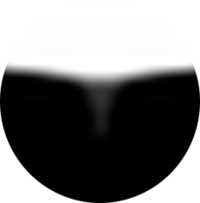}};
\draw (5.4, 0) node [align = center] {\includegraphics[width=2.4cm, keepaspectratio]{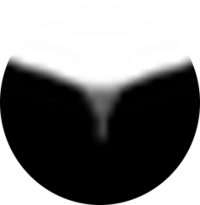}};
\draw (8.1, 0) node [align = center] {\includegraphics[width=2.4cm, keepaspectratio]{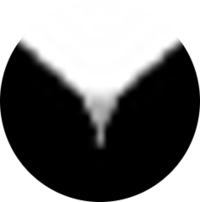}};
\draw [red, ultra thick] (0,0) circle [radius=1.2cm];
\draw [red, ultra thick] (2.7,0) circle [radius=1.2cm];
\draw [red, ultra thick] (5.4,0) circle [radius=1.2cm];
\draw [red, ultra thick] (8.1,0) circle [radius=1.2cm];
\end{tikzpicture}
}
\vspace*{-4.8ex}

\caption{Zoom into marked regions in Figure \ref{fig:circleTriangle}.}
\vspace*{-3ex}
\label{fig:circleTriangleIntensity}
\end{wrapfigure}
Compared to this the piecewise geodesic interpolation is not smooth in time and local features of the input images spread just over two consecutive time intervals.
Note that the metamorphosis B\'ezier curve in Figure\,\ref{fig:circleTriangle} is not solely generated by a transport.
Indeed, the generation of the lower tip of the triangle is not possible via a flow of diffeomorphisms. Hence, the 
triangle tip is generated via intensity modulation (cf. Fig.\,\ref{fig:circleTriangleIntensity}).

\section{Applications}
In this section we present two applications of B\'{e}zier curves in the space of images.
First, we use our approach for the modulation of human face interpolation.
Then we compute discrete B\'{e}zier curves to reconstruct animations of shapes. 

\subsection{Modulation of Interpolation Path}
B\'{e}zier curves can be used to modulate a (geodesic) interpolation path between two given images via prescribing additional control images, which are used to
let additional image pattern appear and spread along a smooth and natural looking path connecting the two end images. 
 Figures \ref{fig:Faces-AE-BW-BG} and \ref{fig:teaser} display a quadratic and a cubic B\'{e}zier curve between different portraits.
 Furthermore the B\'{e}zier curves are compared with piecewise geodesic interpolations between the control images.
The resulting path between the start and the end image ((A,C) and (A,D), respectively) consists of natural looking facial image for all times. 
The resulting path in image space is smooth in time. 
Features of the control images are clearly visible and significantly spread in time over the image sequence.
\Eg in Fig. \ref{fig:Faces-AE-BW-BG} the hairstyle of face $C$ is clearly visible in the images at time steps $2-7$ of the B\'{e}zier curve.
Furthermore, the shape of the chin of face $A$ carries over up to the time step $5-6$ and the moustache from face $B$ spreads over the whole sequence 
from time step $1$ to $7$.
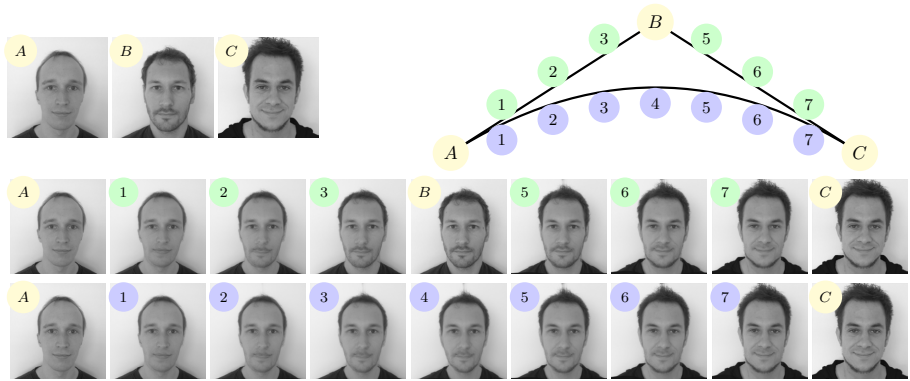
\begin{figure}
\centering
\resizebox{1.0\linewidth}{!}{
\begin{tikzpicture}[x=1cm, y=1cm]
  \draw[-, very thick] (6,0.5) -- (9.875,3) -- (13.75,0.5);
  
\draw (6,0.5) coordinate (A);
\draw (9.875,3) coordinate (B);
\draw (13.75,0.5) coordinate (C);
\draw (103/12,13/6) coordinate (P);
\draw (67/6,13/6) coordinate (Q);
\draw [very thick] (A) .. controls (P) and (Q) .. (C);
 \draw (6.96875,67/64)[below] node[circle, minimum size = 0.05cm, fill=blue!20] {\small$ 1 $};
 \draw (7.9375,23/16)[below] node[circle, minimum size = 0.05cm, fill=blue!20] {\small$2 $};
 \draw (8.90625,107/64)[below] node[circle, minimum size = 0.05cm, fill=blue!20] {\small$3 $};
 \draw (9.875,7/4)[below] node[circle, minimum size = 0.05cm, fill=blue!20] {\small$4 $};
 \draw (10.84375,107/64)[below] node[circle, minimum size = 0.05cm, fill=blue!20] {\small$5 $};
 \draw (11.8125,23/16)[below] node[circle, minimum size = 0.05cm, fill=blue!20] {\small$6 $};
 \draw (12.78125,67/64)[below] node[circle, minimum size = 0.05cm, fill=blue!20] {\small$7 $};
 
 \draw (6.96875,9/8)[above] node[circle, minimum size = 0.05cm, fill=green!20] {\small$1 $};
 \draw (7.9375,7/4)[above] node[circle, minimum size = 0.05cm, fill=green!20] {\small$2 $};
 \draw (8.90625,19/8)[above] node[circle, minimum size = 0.05cm, fill=green!20] {\small$3 $};
 \draw (10.84375,19/8)[above] node[circle, minimum size = 0.05cm, fill=green!20] {\small$5 $};
 \draw (11.8125,7/4)[above] node[circle, minimum size = 0.05cm, fill=green!20] {\small$6 $};
 \draw (12.78125,9/8)[above] node[circle, minimum size = 0.05cm, fill=green!20] {\small$7 $};
\draw (6,0.5) node[circle, minimum size = 0.05cm, fill=yellow!20] {$ A $};
\draw (9.875,3) node[circle, minimum size = 0.05cm, fill=yellow!20] {$ B $};
\draw (13.75,0.5) node[circle, minimum size = 0.05cm, fill=yellow!20] {$ C $};
 
\draw (-1.45 , 1.75 ) node [align = center] {\includegraphics[width=1.9cm, keepaspectratio]{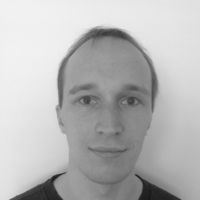}};
\draw (0.55, 1.75) node [align = center] {\includegraphics[width=1.9cm, keepaspectratio]{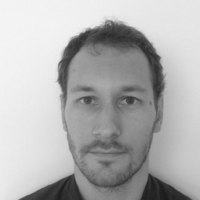}};
\draw (2.55, 1.75) node [align = center] {\includegraphics[width=1.9cm, keepaspectratio]{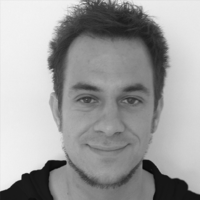}};
\draw (-1.45,-0.9) node [align = center] {\includegraphics[width=1.8cm, keepaspectratio]{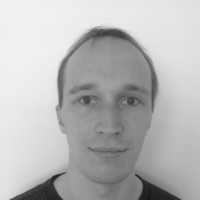}};
\draw (0.45,-0.9) node [align = center] {\includegraphics[width=1.8cm, keepaspectratio]{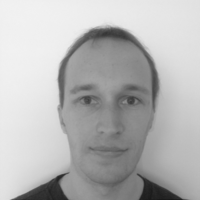}};
\draw (2.35,-0.9) node [align = center] {\includegraphics[width=1.8cm, keepaspectratio]{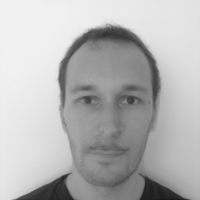}};
\draw (4.25,-0.9) node [align = center] {\includegraphics[width=1.8cm, keepaspectratio]{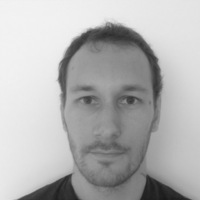}};
\draw (6.15,-0.9) node [align = center] {\includegraphics[width=1.8cm, keepaspectratio]{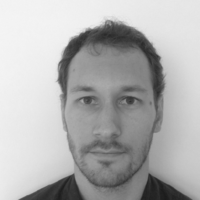}};
\draw (8.05,-0.9) node [align = center] {\includegraphics[width=1.8cm, keepaspectratio]{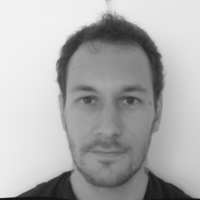}};
\draw (9.95,-0.9) node [align = center] {\includegraphics[width=1.8cm, keepaspectratio]{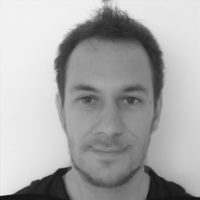}};
\draw (11.85,-0.9) node [align = center] {\includegraphics[width=1.8cm, keepaspectratio]{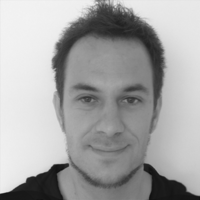}};
\draw (13.75,-0.9) node [align = center] {\includegraphics[width=1.8cm, keepaspectratio]{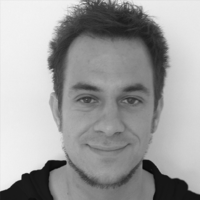}};
\draw (-1.45,-2.9) node [align = center] {\includegraphics[width=1.8cm, keepaspectratio]{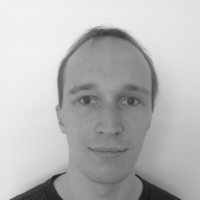}};
\draw (0.45,-2.9) node [align = center] {\includegraphics[width=1.8cm, keepaspectratio]{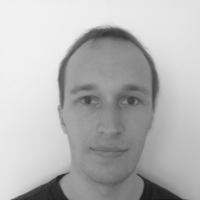}};
\draw (2.35,-2.9) node [align = center] {\includegraphics[width=1.8cm, keepaspectratio]{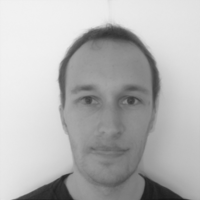}};
\draw (4.25,-2.9) node [align = center] {\includegraphics[width=1.8cm, keepaspectratio]{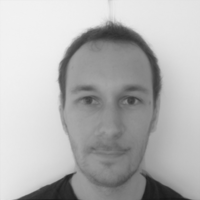}};
\draw (6.15,-2.9) node [align = center] {\includegraphics[width=1.8cm, keepaspectratio]{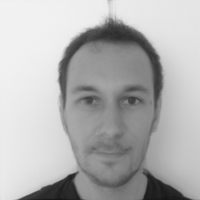}};
\draw (8.05,-2.9) node [align = center] {\includegraphics[width=1.8cm, keepaspectratio]{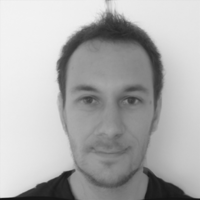}};
\draw (9.95,-2.9) node [align = center] {\includegraphics[width=1.8cm, keepaspectratio]{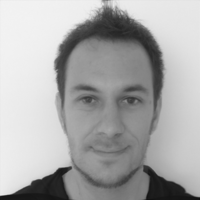}};
\draw (11.85,-2.9) node [align = center] {\includegraphics[width=1.8cm, keepaspectratio]{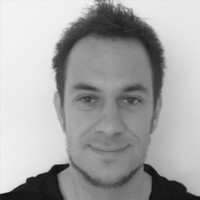}};
\draw (13.75,-2.9) node [align = center] {\includegraphics[width=1.8cm, keepaspectratio]{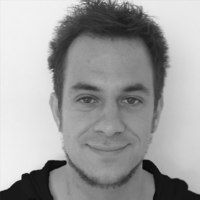}};
\draw (-2.15,2.45) node[circle, minimum size = 0.05cm, fill=yellow!20] {\footnotesize$ A $};
\draw (-0.15,2.45) node[circle, minimum size = 0.05cm, fill=yellow!20] {\footnotesize$ B $};
\draw (1.85,2.45) node[circle, minimum size = 0.05cm, fill=yellow!20] {\footnotesize$ C $};
 \draw (-2.1,-0.25) node[circle, minimum size = 0.05cm, fill=yellow!20] {\footnotesize$ A $};
 \draw (-0.2,-0.25) node[circle, minimum size = 0.05cm, fill=green!20] {\footnotesize$ 1 $};
 \draw (1.7,-0.25) node[circle, minimum size = 0.05cm, fill=green!20] {\footnotesize$2 $};
 \draw (3.6,-0.25) node[circle, minimum size = 0.05cm, fill=green!20] {\footnotesize$3 $};
 \draw (5.5,-0.25) node[circle, minimum size = 0.05cm, fill=yellow!20] {\footnotesize$ B $};
 \draw (7.4,-0.25) node[circle, minimum size = 0.05cm, fill=green!20] {\footnotesize$5 $};
 \draw (9.3,-0.25) node[circle, minimum size = 0.05cm, fill=green!20] {\footnotesize$6 $};
 \draw (11.2,-0.25) node[circle, minimum size = 0.05cm, fill=green!20] {\footnotesize$7 $};
 \draw (13.1,-0.25) node[circle, minimum size = 0.05cm, fill=yellow!20] {\footnotesize$ C $};
 
 \draw (-2.1,-2.25) node[circle, minimum size = 0.05cm, fill=yellow!20] {\footnotesize$ A $};
 \draw (-0.2,-2.25) node[circle, minimum size = 0.05cm, fill=blue!20] {\footnotesize$1 $};
 \draw (1.7,-2.25) node[circle, minimum size = 0.05cm, fill=blue!20] {\footnotesize$2 $};
 \draw (3.6,-2.25) node[circle, minimum size = 0.05cm, fill=blue!20] {\footnotesize$3 $};
 \draw (5.5,-2.25) node[circle, minimum size = 0.05cm, fill=blue!20] {\footnotesize $4$};
 \draw (7.4,-2.25) node[circle, minimum size = 0.05cm, fill=blue!20] {\footnotesize$5 $};
 \draw (9.3,-2.25) node[circle, minimum size = 0.05cm, fill=blue!20] {\footnotesize$6 $};
 \draw (11.2,-2.25) node[circle, minimum size = 0.05cm, fill=blue!20] {\footnotesize$7 $};
 \draw (13.1,-2.25) node[circle, minimum size = 0.05cm, fill=yellow!20] {\footnotesize$ C $};
\end{tikzpicture}
}
\caption{Piecewise discrete geodesic (middle) and quadratic B\'{e}zier curve (bottom) between human faces. $K=8$, $\delta=7.5\cdot10^{-3}$, $\gamma=10^{-3}$.}
\label{fig:Faces-AE-BW-BG}
\end{figure}

\subsection{Image Sketches as Control Shapes for Animation}
Discrete B\'{e}zier curves can also be used to design animation paths in the space of images based on very moderate user interaction.
In fact, to animate an object in an image
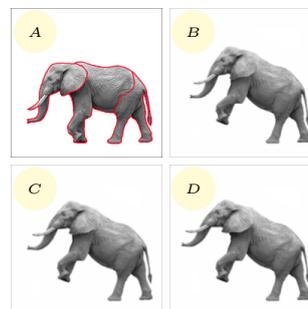
\begin{wrapfigure}{r}{.37\linewidth}
\vspace*{-3ex}
\centering
\resizebox{1.0\linewidth}{!}{
\begin{tikzpicture}[x=1cm, y=1cm]
 
\draw (0 ,0 ) node [align = center] {\includegraphics[width=1.9cm, keepaspectratio]{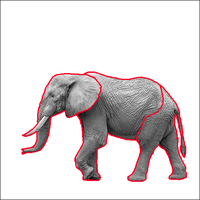}};
\draw (2, 0) node [align = center] {\includegraphics[width=1.9cm, keepaspectratio]{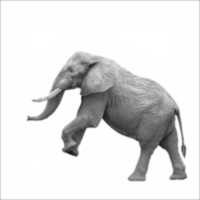}};
\draw (0, -2) node [align = center] {\includegraphics[width=1.9cm, keepaspectratio]{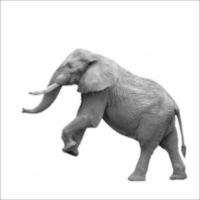}};
\draw (2, -2) node [align = center] {\includegraphics[width=1.9cm, keepaspectratio]{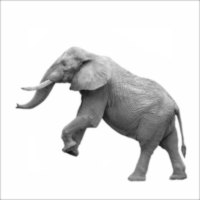}};
\draw (-0.65,0.65) node[circle, minimum size = 0.05cm, fill=yellow!20] {\tiny$ A $};
\draw (1.35,0.65) node[circle, minimum size = 0.05cm, fill=yellow!20] {\tiny$ B $};
\draw (-0.65,-1.35) node[circle, minimum size = 0.05cm, fill=yellow!20] {\tiny$ C $};
\draw (1.35,-1.35) node[circle, minimum size = 0.05cm, fill=yellow!20] {\tiny$ D $};
\end{tikzpicture}
}
\vspace*{-5ex}
\caption{Modeling of control images.}
\vspace*{-3ex}
\label{fig:ElephantSegments}
\end{wrapfigure}
 the user might cut the object into pieces and generate control images via translation, rotation or scaling of the different pieces.
Then, de Casteljau's algorithm is used to generate from these sketches a smooth animation path in the space of images, which approximately recovers the poses described by the control 
images. As a proof of concept we consider in Fig. \ref{fig:ArmSketches} four control images (A-D), which are generated from a very simple model of an arm (top left).
While a geodesic curve between the first (A) and the last image (D) prefers to collapse the forearm (second row), the additional control points (B,C) give further guidance to reconstruct the bending of the arm along a discrete cubic B\'{e}zier curve (third row).
Next, we consider in Fig. \ref{fig:ElephantSegments} a photograph of an elephant cut into three pieces (A), which are reconfigured into four control images 
(A-D). These control images are then taken into account to compute a discrete cubic B\'{e}zier curve (Figure \ref{fig:elephant}).
The resulting sequence represents the turning and stretching of the head and the raising of the forelegs.  
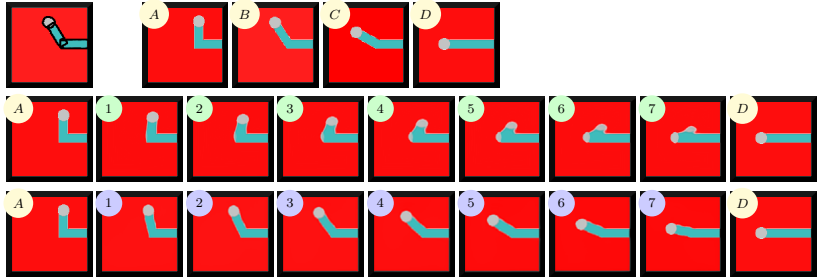
\begin{figure}
\centering
\resizebox{0.9\linewidth}{!}{
\begin{tikzpicture}[x=1cm, y=1cm]
\draw (-1.45 , 1.35 ) node [align = center] {\includegraphics[width=1.8cm, keepaspectratio]{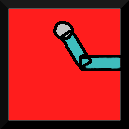}};
\draw (1.4 , 1.35 ) node [align = center] {\includegraphics[width=1.8cm, keepaspectratio]{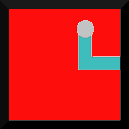}};
\draw (3.3, 1.35) node [align = center] {\includegraphics[width=1.8cm, keepaspectratio]{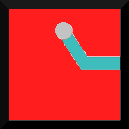}};
\draw (5.2, 1.35) node [align = center] {\includegraphics[width=1.8cm, keepaspectratio]{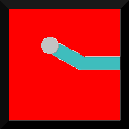}};
\draw (7.1, 1.35) node [align = center] {\includegraphics[width=1.8cm, keepaspectratio]{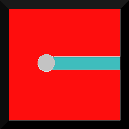}};
\draw (-1.45,-0.65) node [align = center] {\includegraphics[width=1.8cm, keepaspectratio]{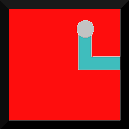}};
\draw (0.45,-0.65) node [align = center] {\includegraphics[width=1.8cm, keepaspectratio]{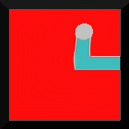}};
\draw (2.35,-0.65) node [align = center] {\includegraphics[width=1.8cm, keepaspectratio]{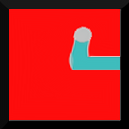}};
\draw (4.25,-0.65) node [align = center] {\includegraphics[width=1.8cm, keepaspectratio]{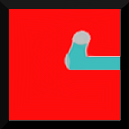}};
\draw (6.15,-0.65) node [align = center] {\includegraphics[width=1.8cm, keepaspectratio]{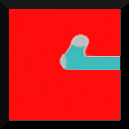}};
\draw (8.05,-0.65) node [align = center] {\includegraphics[width=1.8cm, keepaspectratio]{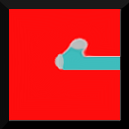}};
\draw (9.95,-0.65) node [align = center] {\includegraphics[width=1.8cm, keepaspectratio]{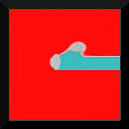}};
\draw (11.85,-0.65) node [align = center] {\includegraphics[width=1.8cm, keepaspectratio]{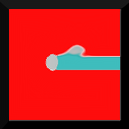}};
\draw (13.75,-0.65) node [align = center] {\includegraphics[width=1.8cm, keepaspectratio]{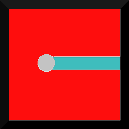}};
\draw (-1.45,-2.65) node [align = center] {\includegraphics[width=1.8cm, keepaspectratio]{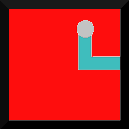}};
\draw (0.45,-2.65) node [align = center] {\includegraphics[width=1.8cm, keepaspectratio]{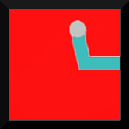}};
\draw (2.35,-2.65) node [align = center] {\includegraphics[width=1.8cm, keepaspectratio]{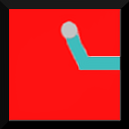}};
\draw (4.25,-2.65) node [align = center] {\includegraphics[width=1.8cm, keepaspectratio]{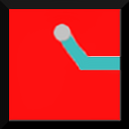}};
\draw (6.15,-2.65) node [align = center] {\includegraphics[width=1.8cm, keepaspectratio]{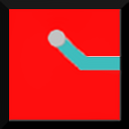}};
\draw (8.05,-2.65) node [align = center] {\includegraphics[width=1.8cm, keepaspectratio]{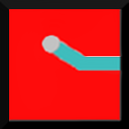}};
\draw (9.95,-2.65) node [align = center] {\includegraphics[width=1.8cm, keepaspectratio]{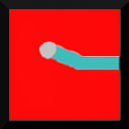}};
\draw (11.85,-2.65) node [align = center] {\includegraphics[width=1.8cm, keepaspectratio]{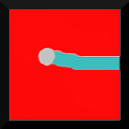}};
\draw (13.75,-2.65) node [align = center] {\includegraphics[width=1.8cm, keepaspectratio]{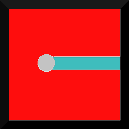}};
\draw (0.75,2) node[circle, minimum size = 0.05cm, fill=yellow!20] {\footnotesize$ A $};
\draw (2.65,2) node[circle, minimum size = 0.05cm, fill=yellow!20] {\footnotesize$ B $};
\draw (4.55,2) node[circle, minimum size = 0.05cm, fill=yellow!20] {\footnotesize$ C $};
\draw (6.45,2) node[circle, minimum size = 0.05cm, fill=yellow!20] {\footnotesize$ D $};
 \draw (-2.1,0) node[circle, minimum size = 0.05cm, fill=yellow!20] {\footnotesize$ A $};
 \draw (-0.2,0) node[circle, minimum size = 0.05cm, fill=green!20] {\footnotesize$ 1 $};
 \draw (1.7,0) node[circle, minimum size = 0.05cm, fill=green!20] {\footnotesize$2 $};
 \draw (3.6,0) node[circle, minimum size = 0.05cm, fill=green!20] {\footnotesize$3 $};
 \draw (5.5,0) node[circle, minimum size = 0.05cm, fill=green!20] {\footnotesize$ 4 $};
 \draw (7.4,0) node[circle, minimum size = 0.05cm, fill=green!20] {\footnotesize$5 $};
 \draw (9.3,0) node[circle, minimum size = 0.05cm, fill=green!20] {\footnotesize$6 $};
 \draw (11.2,0) node[circle, minimum size = 0.05cm, fill=green!20] {\footnotesize$7 $};
 \draw (13.1,0) node[circle, minimum size = 0.05cm, fill=yellow!20] {\footnotesize$ D $};
 
 \draw (-2.1,-2) node[circle, minimum size = 0.05cm, fill=yellow!20] {\footnotesize$ A $};
 \draw (-0.2,-2) node[circle, minimum size = 0.05cm, fill=blue!20] {\footnotesize$1 $};
 \draw (1.7,-2) node[circle, minimum size = 0.05cm, fill=blue!20] {\footnotesize$2 $};
 \draw (3.6,-2) node[circle, minimum size = 0.05cm, fill=blue!20] {\footnotesize$3 $};
 \draw (5.5,-2) node[circle, minimum size = 0.05cm, fill=blue!20] {\footnotesize $4$};
 \draw (7.4,-2) node[circle, minimum size = 0.05cm, fill=blue!20] {\footnotesize$5 $};
 \draw (9.3,-2) node[circle, minimum size = 0.05cm, fill=blue!20] {\footnotesize$6 $};
 \draw (11.2,-2) node[circle, minimum size = 0.05cm, fill=blue!20] {\footnotesize$7 $};
 \draw (13.1,-2) node[circle, minimum size = 0.05cm, fill=yellow!20] {\footnotesize$ D $};
\end{tikzpicture}
}
\caption{A simple model of an arm (top left) is split up into three pieces, they are reconfigured into four control images (A-D).
A geodesic interpolation of the images (A) and (D) (middle row) and a discrete cubic B\'{e}zier curve with $K=8$, $\delta = 10^{-2}$, $\gamma = 10^{-3}$ is computed (bottom row).}
\label{fig:ArmSketches}
\end{figure}

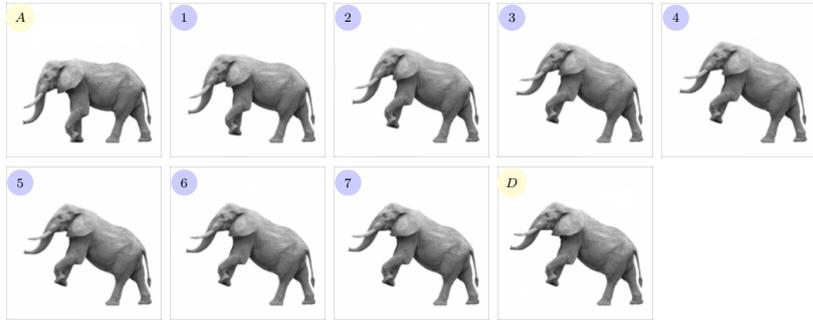
\begin{figure}
\centering
\resizebox{0.9\linewidth}{!}{
\begin{tikzpicture}[x=1cm, y=1cm]
\draw (0,0) node [align = center] {\includegraphics[width=3.3cm, keepaspectratio]{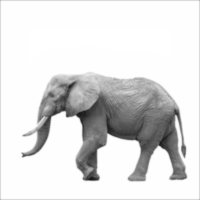}};
\draw (3.5,0) node [align = center] {\includegraphics[width=3.3cm, keepaspectratio]{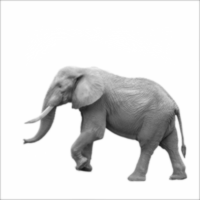}};
\draw (7,0) node [align = center] {\includegraphics[width=3.3cm, keepaspectratio]{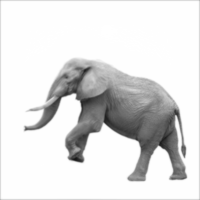}};
\draw (10.5,0) node [align = center] {\includegraphics[width=3.3cm, keepaspectratio]{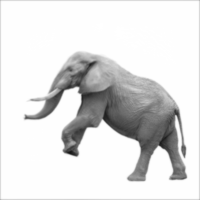}};
\draw (14,0) node [align = center] {\includegraphics[width=3.3cm, keepaspectratio]{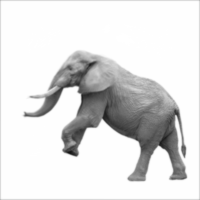}};

\draw (0,-3.5) node [align = center] {\includegraphics[width=3.3cm, keepaspectratio]{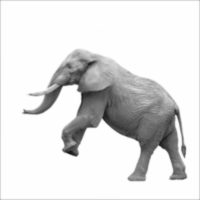}};
\draw (3.5,-3.5) node [align = center] {\includegraphics[width=3.3cm, keepaspectratio]{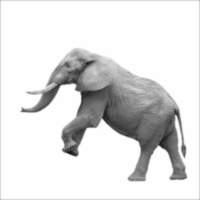}};
\draw (7,-3.5) node [align = center] {\includegraphics[width=3.3cm, keepaspectratio]{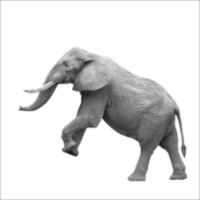}};
\draw (10.5,-3.5) node [align = center] {\includegraphics[width=3.3cm, keepaspectratio]{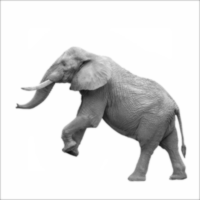}};
 
 \draw (-1.35,1.35) node[circle, minimum size = 0.05cm, fill=yellow!20] {\footnotesize$ A $};
 \draw (2.15,1.35) node[circle, minimum size = 0.05cm, fill=blue!20] {\footnotesize$1 $};
 \draw (5.65,1.35) node[circle, minimum size = 0.05cm, fill=blue!20] {\footnotesize$2 $};
 \draw (9.15,1.35) node[circle, minimum size = 0.05cm, fill=blue!20] {\footnotesize$3 $};
 \draw (12.65,1.35) node[circle, minimum size = 0.05cm, fill=blue!20] {\footnotesize $4$};
 
 \draw (-1.35,-2.2) node[circle, minimum size = 0.05cm, fill=blue!20] {\footnotesize$5 $};
 \draw (2.15,-2.2) node[circle, minimum size = 0.05cm, fill=blue!20] {\footnotesize$6 $};
 \draw (5.65,-2.2) node[circle, minimum size = 0.05cm, fill=blue!20] {\footnotesize$7 $};
 \draw (9.15,-2.2) node[circle, minimum size = 0.05cm, fill=yellow!20] {\footnotesize$ D $};
\end{tikzpicture}
}
\caption{An animation sequence based on the image of an elephant (www.fotolia.com \copyright eyetronic) via discrete cubic B\'{e}zier curve is shown for $K=8$, $\delta=4\cdot10^{-3}$, $\gamma=10^{-3}$.
The underlying control images are depicted in Fig. \ref{fig:ElephantSegments}.}
\label{fig:elephant}
\end{figure}

\section{Conclusions}
We have defined B\'{e}zier curves on the image manifold equipped with the metamorphosis metric.
Based on the notion of a discrete metamorphosis path energy and a corresponding discrete geodesic interpolation,
de Casteljau's algorithm allows the robust and efficient computation of 
discrete B\'{e}zier curves in the space of images. 
Striking features of this approach are the smoothness in time and the global impact of features of the control images. 
This approach can be regarded as a conceptual study for curve modeling on shape spaces.
Indeed, B\'{e}zier curves are just one prominent example for classes of curves constructed via hierarchical convex combination 
along straight lines.  Perspective future generalizations include the Neville--Aitken scheme for general polynomials and B-Spline curves.


\begin{thebibliography}{4}

\bibitem[AK98]{ArKh98}
V.~Arnold and B.~Khesin.
\newblock {\em Topological methods in hydrodynamics}.
\newblock Springer, 1998.

\bibitem[Arn66]{Ar66a}
Vladimir Arnold.
\newblock Sur la g\'eom\'etrie diff\'erentielle des groupes de lie de dimension
  infinie et ses applications \`a l'hydrodynamique des fluides parfaits.
\newblock {\em Annales de l'institut Fourier}, 16:319--361, 1966.

\bibitem[BB00]{BeBr00}
Jean-David Benamou and Yann Brenier.
\newblock A computational fluid mechanics solution to the {M}onge-{K}antorovich
  mass transfer problem.
\newblock {\em Numer. Math.}, 84(3):375--393, 2000.

\bibitem[BER15]{TimeDiscrete}
B.~{Berkels}, A.~{Effland}, and M.~{Rumpf}.
\newblock {Time Discrete Geodesic Paths in the Space of Images}.
\newblock {\em ArXiv e-prints}, March 2015.

\bibitem[BMTY02]{BeMiTrYo02}
M.~F. Beg, M.I. Miller, A.~Trouv{\'{e}}, and L.~Younes.
\newblock Computational anatomy: Computing metrics on anatomical shapes.
\newblock In {\em Proceedings of 2002 IEEE ISBI}, pages 341--344, 2002.

\bibitem[CSC01]{CaSiCr01}
M.\ Camarinha, F. Silva Leite, and P.\ Crouch.
\newblock On the geometry of {R}iemannian cubic polynomials.
\newblock {\em Differential Geom. Appl.}, 15(2):107--135, 2001.

\bibitem[DGM98]{DuGrMi98}
D.~Dupuis, U.~Grenander, and M.I. Miller.
\newblock Variational problems on flows of diffeomorphisms for image matching.
\newblock {\em Quarterly of Applied Mathematics}, 56:587--600, 1998.

\bibitem[FLPJ04]{FlLuPi04}
P.T. Fletcher, Conglin Lu, S.M. Pizer, and Sarang Joshi.
\newblock Principal geodesic analysis for the study of nonlinear statistics of
  shape.
\newblock {\em Medical Imaging, IEEE Transactions on}, 23(8):995--1005, 2004.

\bibitem[GSA14]{GoSaAb14}
P.-Y.\ Gousenbourger, C.\ Samir, and P.-A.\ Absil.
\newblock Piecewise-Bezier $C^1$ interpolation on Riemannian manifolds with application to 2D shape morphing.
\newblock {\em Proceedings of ICPR 2014}, 2014.

\bibitem[MCV08]{MoCaVe08}
D.M.\ Morera, P.C.\ Carvalho, L.\ Velho.
\newblock Modeling on triangulations with geodesic curves.
\newblock {\em The Visual Computer}, 24(12):1025--1037, 2008.

\bibitem[MY01]{MiYo01}
M.~I. Miller and L.~Younes.
\newblock Group actions, homeomorphisms, and matching: a general framework.
\newblock {\em International Journal of Computer Vision}, 41(1--2):61--84,
  2001.

\bibitem[Nv91]{NeSi91}
J.~Ne\v{c}as and M.~\v{S}ilhav\'y.
\newblock Multipolar viscous fluids.
\newblock {\em Quarterly of Applied Mathematics}, 49(2):247--265, 1991.

\bibitem[PR95]{PaRa95}
F.C.\ Park and B.\ Ravani.
\newblock B\'ezier curves on {R}iemannian manifolds and {L}ie groups with kinematics applications.
\newblock {\em J. Mech. Des.}, 117(1):36--40, 1995.

\bibitem[PN07a]{PoNo07a}
T.\ Popiel and L.\ Noakes.
\newblock Elastica in {$\rm SO(3)$}.
\newblock {\em J. Aust. Math. Soc.}, 83(1):105--124, 2007.

\bibitem[PN07b]{PoNo07}
T.\ Popiel and L.\ Noakes.
\newblock B\'ezier curves and {$C^2$} interpolation in {R}iemannian manifolds.
\newblock {\em J. Approx. Theory}, 148(2):111--127, 2007.

\bibitem[RW14]{RuWi12b}
Martin Rumpf and Benedikt Wirth.
\newblock Variational time discretization of geodesic calculus.
\newblock {\em IMA Journal of Numerical Analysis}, 2014.
\newblock (to appear).

\bibitem[TY05a]{TrYo05a}
Alain Trouv\'e and Laurent Younes.
\newblock Local geometry of deformable templates.
\newblock {\em SIAM J. MATH. ANAL}, 37(2):17--59, 2005.

\bibitem[TY05b]{TrYo05}
Alain Trouv\'e and Laurent Younes.
\newblock Metamorphoses through {L}ie group action.
\newblock {\em Foundations of Computational Mathematics}, 5(2):173--198, 2005.

\bibitem[ZYHT07]{ZhYaHa07}
Lei Zhu, Yan Yang, Steven Haker, and Allen Tannenbaum.
\newblock An image morphing technique based on optimal mass preserving mapping.
\newblock {\em IEEE Transactions on Image Processing}, 16(6):1481--1495, 2007.

\end{thebibliography}
\end{document}